\documentclass[12pt,centertags,oneside]{amsart}
\usepackage{amsmath,amstext,amsthm,amscd,typearea,hyperref}
\usepackage{amssymb}
\usepackage{a4wide}
\usepackage[mathscr]{eucal}
\usepackage{mathrsfs}
\usepackage{typearea}
\usepackage{charter}
\usepackage{pdfsync}
\usepackage{stmaryrd}
 \usepackage[T1]{fontenc}
\usepackage{enumitem}

\usepackage{calc}

\usepackage{changepage}

\usepackage[a4paper,width=16.5cm,top=3cm,bottom=3cm]{geometry}

\numberwithin{equation}{section}

\usepackage[french,english]{babel}

\allowdisplaybreaks
\emergencystretch=\maxdimen
\usepackage{multicol}

\usepackage{xcolor}

\newtheorem{theorem}{Theorem}[section]

\newtheorem{proposition}[theorem]{Proposition}
\newtheorem{corollary}[theorem]{Corollary}
\newtheorem{lemma}[theorem]{Lemma}
\newtheorem{remark}[theorem]{Remark}

\newtheorem*{definition*}{Definition}

\newcommand{\cali}[1]{\mathscr{#1}}

\newcommand{\supp}{{\rm supp}}

\newcommand{\dist}{\mathop{\mathrm{dist}}\nolimits}

\newcommand{\ddc}{{\rm dd^c}}
\newcommand{\dc}{{\rm d^c}}

\newcommand{\dd}{{\rm d}}

\newcommand{\dbar}{\overline\partial}

\newcommand{\SH}{{\rm SH}}

\newcommand{\vep}{\varepsilon}

\newcommand{\vol}{{\rm vol}}

\newcommand{\Area}{{\rm Area}}

\newcommand{\Cc}{\cali{C}}

\newcommand{\Fc}{\cali{F}}

\newcommand{\Oc}{\cali{O}}

\newcommand{\W}{\cali{W}}

\newcommand{\cM}{\mathcal{M}}

\newcommand{\B}{\mathbb{B}}

\newcommand{\C}{\mathbb{C}}

\newcommand{\N}{\mathbb{N}}
\newcommand{\Z}{\mathbb{Z}}
\newcommand{\R}{\mathbb{R}}

\newcommand{\lp}{\langle}
\newcommand{\rp}{\rangle}

\newcommand{\norm}[1]{\lVert#1\rVert}

\newcommand{\oA}{\mathcal{A}}
\newcommand{\oB}{\mathcal{B}}

\newcommand{\oR}{\mathcal{R}}
\newcommand{\oE}{\mathcal{E}}

\newcommand{\oL}{\mathcal{L}}
\newcommand{\diam}{{\rm diam}}
\newcommand{\oI}{\mathcal{I}}
\newcommand{\oJ}{\mathcal{J}}

\newcommand{\oD}{\mathcal{D}}

\newcommand{\oK}{\mathcal{K}}

\newcommand{\fh}{\mathfrak{h}}



\usepackage{fancyhdr}
\title{Growth Rate of balls of holomorphic sections\\ on compact Riemann surfaces}

\author{Hao Wu}
\address{School of Mathematics,  Nanjing University - Nanjing - China 210093}
\email{haowu@nju.edu.cn}

\date{}
\thanks{}

\begin{document}

\begin{abstract}
Let $X$ be a compact Riemann surface and let $L$ be a positive line bundle on $X$. We obtain the growth speed of unit ball volume in $H^0(X,L^n)$ towards the energy at equilibrium. As an application, we also obtain the speed of Fekete measures converging to the equilibrium measure.
\end{abstract}

\clearpage\maketitle
\thispagestyle{empty}


\medskip


\setcounter{tocdepth}{1}
\tableofcontents\normalfont

\section{Introduction}

Let $X$ be a compact Riemann surface and let $L$ be a positive holomorphic line bundle on $X$ with $\deg(L) \geq 1$.  Fix a Hermitian metric $h$ on $L$ such that the curvature form $c_1(L,h)$ is strictly positive. The normalized $(1,1)$-form
$$ \omega :=c_1(L,h)/\deg(L)$$
is a smooth probability measure on $X$.

We call \textit{weighted subset of $X$} a data $(K,\phi,\mu)$, where $K$ is a compact subset of $X$, $\phi$ is a continuous function on $X$ and $\mu$ is a probability measure  with $\supp(\mu)\subset  K$.

For each $n\in\N$, consider the space of global holomorphic sections $H^0 (X,L^n)$ of $L^n$.  Its dimension  is
\[
N_n:=\dim_{\mathbb{C}} H^0(X, L^n) = n \cdot \deg(L) - g + 1
\] 
by Riemann-Roch theorem, where  $g$ is the genus of $X$.

For each element $s\in  H^0(X, L^n)$, the norm at each point $x\in X$, induced by $h$, is denoted by $|s(x)|_h$.
The weighted subset $(K,\phi,\mu)$ gives  a new metric
$$ |s(x)|_{n\phi}:=|s(x)|_h e^{-n\phi}.$$
After that, we   define two norms on $H^0 (X,L^n)$:
\begin{equation*}
\norm{s}_{L^\infty (K,n\phi)}:=\sup _{x\in K} |s(x)|_{n\phi}
\end{equation*}
and
\begin{equation*}\label{defn-l2metric}
 \norm{s}_{L^2 (\mu,n\phi)}^2:=  \int_X |s(x)|_{n\phi}^2 \,\dd \mu(x)= \int_K |s(x)|_{n\phi}^2 \,\dd \mu(x). 
\end{equation*}

Let $\oB^\infty_n(K,\phi)$ (resp.\ $\oB^2_n(\mu,\phi)$) denote the unit ball in $H^0 (X,L^n)$ with respect the $L^\infty (K,n\phi)$-norm (resp.\ $L^2(\mu,n\phi)$-norm). Define 
$$   \oL_n(K,\phi):={1\over n N_n} \log \vol \oB^\infty_n (K,\phi)$$
and
$$   \oL_n(\mu,\phi):={1\over n N_n} \log \vol \oB^2_n (\mu,\phi).$$
 Observe that $\oL_n(K,\phi)\leq \oL_n(\mu,\phi)$.  
Here, $\vol$ denotes the Lebesgue measure on the vector space $H^0 (X,L^n)$, which is only defined up to a multiplicative constant.  Nevertheless, the differences $\oL_n(K_1,\phi_1)-\oL_n(K_2,\phi_2)$ and   $\oL_n(\mu_1,\phi_1)-\oL_n(\mu_2,\phi_2)$ are well-defined for any two weighted compact subsets $(K_1,\phi_1,\mu_1)$ and $(K_2,\phi_2,\mu_2)$.

\medskip

A function $\varphi$ on $X$ with values in $\mathbb{R} \cup \{-\infty\}$ is called \textit{quasi-subharmonic} if, locally, it can be written as the difference of a subharmonic function and a smooth function. If $\varphi$ is quasi-subharmonic, then there exists a constant $c \geq 0$ such that $\ddc \varphi \geq -c  \omega $ in the sense of currents ($\dc:=\frac{i}{2\pi}(\dbar-\partial)$ and $\ddc=\frac{i}{\pi}\partial\dbar$). When $c = 1$, $\varphi$ is called an  \textit{$\omega $-subharmonic function}, and $\ddc \varphi + \omega $ is a probability measure on $X$ by Stokes' formula. The set of $\omega $-subharmonic functions is denoted by $\SH(X,\omega )$.

For any probability measure $\nu$ on $X$, we can write $\nu = \omega  + \ddc U_\nu$, where $U_\nu$ is the unique $\omega $-subharmonic function such that $\max_X U_\nu = 0$. We call $U_\nu$ the \textit{maximum normalized $\omega $-potential} of $\nu$.
There is an alternative way to normalize the potential $U^\star_\nu$ by requiring that $\int_X U^\star_\nu \, \omega  = 0$. We call $U^\star_\nu$ the \textit{mean normalized  $\omega$-potential}   of $\nu$. By definition, $U_\nu=U^\star_\nu-\max_X U^\star_\nu$.

For every probability measure $\nu$ on $X$,   define the \textit{energy functional}
$$   \oI_\phi(\nu):=-\int_X  U_\nu^\star \,\dd \nu + 2\int_X \phi \,\dd \nu.    $$
Denote by $\cM(K)$ the set of all probability measures on $K$.  The functional $\oI_\phi$ over $\cM(K)$ admits a unique minimizer, denoted by $\nu_{K,\phi}$ (see Section \ref{sec-min}).

\medskip

If $K$ is non-pluripolar,
Berman-Boucksom \cite{ber-bou-invent} proved that  as $n\to \infty$,
$$\oL_n(K_1,\phi_1)-\oL_n(K_2,\phi_2)  \longrightarrow  \min_{\cM(K_1)}\oI_{\phi_1}  -\min_{\cM(K_2)}\oI_{\phi_2}.$$
If moreover both $(K_1,\phi_1,\mu_1)$ and $(K_2,\phi_2,\mu_2)$ both satisfy the Bernstein-Markov condition (c.f.\ Section \ref{sec-BM}), then
$$\oL_n(\mu_1,\phi_1)-\oL_n(\mu_2,\phi_2)  \longrightarrow  \min_{\cM(K_1)}\oI_{\phi_1}  -\min_{\cM(K_2)}\oI_{\phi_2}.   $$
 
Even under the hypothesis $K_1=K_2=X,\phi_1,\phi_2\in\Cc^3(X)$, the speed of the above convergence was unknown until the work of Dinh-Ma-Nguy\^{e}n \cite{din-ma-ngu-ens}, where a rate  $O(n^{-1/2}\log^{3/2}n)$ was obtained.
Our main results of this article  provide the speed $O(n^{-1}\log n)$ when $\phi_1,\phi_2$ are only assumed to be H\"older continuous. Moreover, it seems to be nearby optimal.

\begin{theorem} \label{maintheorem1}
	Let $(K_1,\phi_1,\mu_1)$ and $(K_2,\phi_2,\mu_2)$ be  weighted  subsets of $X$. Suppose $K_1,K_2$ are the closure of non-empty open subsets with $\Cc^2$ boundary and $\phi_1,\phi_2$ are $\gamma$-H\"older for some $0<\gamma\leq 1$. Then there exists a constant $C>0$ independent of $\phi_1,\phi_2$ and $n$, but depending on $K_1,K_2,\norm{\phi_1}_{\Cc^\gamma},\norm{\phi_2}_{\Cc^\gamma}$, such that for every $n>1$, 
	$$  \Big|\oL_n(K_1,\phi_1)-\oL_n(K_2,\phi_2) -  \min_{\cM(K_1)}\oI_{\phi_1}  +\min_{\cM(K_2)}\oI_{\phi_2}  \Big|\leq  C{\log n\over n}.         $$
\end{theorem}

If moreover, the two weighted subsets satisfy  the strong Bernstein-Markov condition (c.f.\ Section \ref{sec-BM}), then we also have the convergence speed for volume of the unit ball in $H^0(X,L^n)$ with respect to  $L^2(\mu,n\phi)$-norm.

\begin{theorem} \label{maintheorem2}
	Let $(K_1,\phi_1,\mu_1)$ and $(K_2,\phi_2,\mu_2)$ be  weighted  subsets of $X$. Suppose $K_1,K_2$ are the closure of non-empty open subsets with $\Cc^2$ boundary and $\phi_1,\phi_2$ are $\gamma$-H\"older  for some $0<\gamma\leq 1$.  If both $(K_1,\phi_1,\mu_1)$ and $(K_2,\phi_2,\mu_2)$  satisfy the strong Bernstein-Markov condition, then there exists a constant $C>0$ independent of $\phi_1,\phi_2$ and $n$, but depending on $K_1,K_2,\norm{\phi_1}_{\Cc^\gamma},\norm{\phi_2}_{\Cc^\gamma}$, such that for every $n>1$, 
	$$  \Big|\oL_n(\mu_1,\phi_1)-\oL_n(\mu_2,\phi_2) -  \min_{\cM(K_1)}\oI_{\phi_1}  +\min_{\cM(K_2)}\oI_{\phi_2} \Big|\leq  C{\log n\over n} .        $$
\end{theorem}

Lastly, we state a condition on $\mu_1,\mu_2$ that is easy to verify and under which the convergence in last theorem holds. We say that $\mu$ satisfies the \textit{mass-density} condition on $K$, if there are two positive constants $c$ and $\tau$ such that for any $x\in K,0<r<1$, 
$$ \mu\big((\B(x,r)\big)\geq cr^\tau.   $$ 
Obviously, every strictly positive $(1,1)$-form satisfies the mass-density condition on any subset. We say the weighted subset $(K,\phi,\mu)$ satisfies the \textit{mass-density} condition if $\mu$ satisfies the mass-density condition on $K$.

\begin{theorem} \label{maintheorem3}
	Let $(K_1,\phi_1,\mu_1)$ and $(K_2,\phi_2,\mu_2)$ be  weighted  subsets of $X$. Suppose $K_1,K_2$ are the closure of non-empty open subsets with $\Cc^2$ boundary and $\phi_1,\phi_2$ are $\gamma$-H\"older  for some $0<\gamma\leq 1$.  If both $(K_1,\phi_1,\mu_1)$ and $(K_2,\phi_2,\mu_2)$ satisfy the mass-density condition, then there exists a constant $C>0$ independent of $\phi_1,\phi_2$ and $n$, but depending on $K_1,K_2,\norm{\phi_1}_{\Cc^\gamma},\norm{\phi_2}_{\Cc^\gamma}$, such that for every $n>1$, 
	$$  \Big|\oL_n(\mu_1,\phi_1)-\oL_n(\mu_2,\phi_2) -  \min_{\cM(K_1)}\oI_{\phi_1}  +\min_{\cM(K_2)}\oI_{\phi_2} \Big|\leq  C{\log n\over n} .        $$
\end{theorem}

The above theorems improve the result of  Dinh-Ma-Nguy\^{e}n \cite[Proposition 3.10]{din-ma-ngu-ens} for the Riemann surfaces case, where their result is valid on any K\"ahler manifold.  As an application, we can also obtain the convergence speed for Fekete measures we introduce now.

Fix a basis $S_n:=(s_1,s_2,\dots,s_{N_n})$  of $H^0 (X,L^n)$.  For any points $x_1,x_2,\dots,x_{N_n}\in X$, define
$$  \det S_n (x_1,x_2,\dots,x_{N_n}):= \det (s_j (x_\ell))_{1\leq j,\ell\leq N_n},$$
which is a section of $(L^n)^{\boxtimes N_n}$ over $X^{N_n}$. The metric $h$ induces in a canonical way a metric  $(h^n)^{\boxtimes N_n}$ on $(L^n)^{\boxtimes N_n}$. Denote this norm by $\big|\det S_n(x_1,\dots,x_{N_n})\big|_h$.
For any weight subset $(K,\phi,\mu)$, we define
$$ \norm{\det S_n}_{L^\infty(K,n\phi)}:=\sup_{(x_1,\dots,x_{N_n})\in K^{N_n}} \big|\det S_n (x_1,\dots,x_{N_n})\big|_h e^{-n\phi(x_1)-\cdots-n\phi(x_{N_n})}  $$
and
$$ \norm{\det S_n}_{L^2(\mu,n\phi)}^2:=\int_{ X^{N_n}} \big|\det S_n (x_1,\dots,x_{N_n})\big|_h ^2 e^{-2n\phi(x_1)-\cdots-2n\phi(x_{N_n})} \,\dd\mu(x_1)\cdots \dd\mu(x_{N_n}) .$$

A point $\Fc_n:=(x_{n,1},\dots,x_{n,N_n})\in K^{N_n}$ is said to be \textit{an $n$-th Fekete
	configuration} for $(K,\phi,\mu)$ if 
$$\big|\det S_n (x_{n,1},\dots,x_{N_n})\big|_h e^{-n\phi(x_1)-\cdots-n\phi(x_{N_n})}=  \norm{\det S_n}_{L^\infty(K,n\phi)}.$$
The existence of such $\Fc_n$ is guaranteed by the compactness of $X$ and it does not depends on the choice of the basis $S_n$. But
it may not be unique. The induced empirical probability measure
$$ \delta_{\Fc_n}:={1\over N_n}\sum_{1\leq j\leq N_n} \delta_{x_{n,j}}      .$$ is called \textit{an $n$-th Fekete measure}.

Berman-Boucksom-Witt \cite{ber-acta} proved that $\delta_{\Fc_n}$ converges weakly to $\nu_{K,\phi}$, settling a long-standing open conjecture, where such problem first appeared in the work of Leja \cite{leja1,leja2},  see also the survey \cite{lev-survey} and the reference therein.

By repeating the strategy in \cite{din-ma-ngu-ens} and using Theorem \ref{maintheorem2}, we obtain the following estimate for the Fekete measures, which also seems to be nearby optimal.

\begin{corollary}
Let $(K,\phi,\mu)$ be a weighted subset of $X$ and $\delta_{\Fc_n}$ be an $n$-th Fekete measure associated with $(K,\phi,\mu)$. Suppose $K$ is the closure of a non-empty  open subset with $\Cc^2$ boundary and $\phi$ is $\gamma$-H\"older  for some $0<\gamma\leq 1$.
There exists a constant $C>0$, such that for any  $\Cc^2$ test function $\psi$, we have 
$$\Big| \int_X \psi \,\dd \delta_{\Fc_n}-\int_X \psi \,\dd \nu_{K,\phi}\Big|\leq C \sqrt{\log n\over n}\,\norm{\psi}_{\Cc^2}.$$
\end{corollary}

Dinh-Ma-Nguy\^{e}n \cite{din-ma-ngu-ens} showed the above convergence holds for any K\"ahler manifold, with a bit slower speed.
Such a convergence has been also obtained by Lev and Ortega-Cerd\`a in \cite{lev-jems}, with an optimal $O(n^{-1/2})$ rate, but only for the special case $K=X,\phi=\mathbf 0$. See also \cite{ahn-mathz,ame-jfa,dinh-tams,duj-book,vu-amj} for some recent works on the  Fekete measures.

\medskip

We now explain briefly our main idea of this article.
Instead of using the approach on Bergman kernel (c.f.\ \cite{auv-ma-mann,cat-ber,co-lu-adv,dinh-jfa1,ma-book,ma-mar-mann,mar-sav-mann,mar-vu-crelle,tian-jdg,zel-imrn-sze}), we apply a  bosonization formula for $|\det S_n|_h$. That allows us to decompose $|\det S_n|_h$ into the product of the norm of a Riemann-theta section and a dominated term:   exponential of summing $N_n(N_n-1)$ Green functions. Most of this article is devoted to prove the upper bound of the dominated term. Moreover, to show the bound is optimal, we need to find exactly where are the maximums appearing. To archive this, we borrow a clever method from  Nishry-Wennman\cite{Nishry-Wennman} in probability theory, defining an auxiliary functional to connect the energy functional and discrete energy functional.

The same approach should also be able to apply in other problems, such as the large deviation principle for beta-ensembles, see e.g.\  \cite{ber-cmp,dinh-tams,gho-nis-con}. To keep the present paper in a reasonable length, we postpone them to our future works.

\medskip
\noindent
\textbf{Notations:} We write   $f\leq O(g)$ if $f\leq C |g|$ for some $C>0$.  The dependence of this constant $C$ on certain parameters, or lack thereof, will be clear from the context. We denote  by $\dist$ the distance on $X$ induced by $\omega $ and write   $\B(x,r)$ the open ball with center $x$ and radius $r$. We use $\partial$ to denote the boundary.

\medskip
\noindent\textbf{Acknowledgements.}  
Thank Professor Tien-Cuong Dinh for his valuable remarks for helping to improve the introduction.

\section{Abel-Jacobi Theory and Green function}

\subsection{Jacobian variety}\label{sec-jac}
We fix a family $\alpha_1,\dots,\alpha_g,\beta_1,\dots,\beta_g$ of $1$-cycles in $X$ which induces a canonical basis for $H_1(X,\Z)$ in the sense that $\alpha_j$ intersects $\beta_j$ once positively for $j=1,\ldots,g$, and any other pair among these cycles does not intersect. Let $\{\phi_1,\dots,\phi_g\}$ be a basis for the complex vector space of holomorphic differential $1$-forms on $X$, and we denote by $\Phi$ the column vector of $g$ entires $\phi_1,\ldots,\phi_g$. The period matrix $\Omega$ is the $g\times 2g$ matrix defined by 
$$\Omega:=\Big(\int_{\alpha_1} \Phi, \,\dots\,,  \int_{\alpha_g} \Phi  , \, \int_{\beta_1}\Phi,\,\dots\, ,\, \int_{\beta_g} \Phi\Big).$$
One can choose the basis $\Phi$ so that $\Omega=(I,\Omega')$, where $I$ is the $g\times g$ identity matrix and $\Omega'$ is a $g\times g$ symmetric matrix having positive definite imaginary part. 

The $2g$ columns of  $\Omega$ are linearly independent over $\R$. Thus, they generate a lattice $\Lambda$ in $\C^g$. We define the \textit{Jacobian variety} of $X$ by
$$\mathrm{Jac}(X):= \C^g/ \Lambda,$$
which is a $g$-dimensional complex torus. Fix a base point $p_\star\in X$ throughout this article. The Abel-Jacobi map $\oA:X \to \mathrm{Jac}(X)$ associated to $p_\star$ is defined by
$$\oA(x):=\Big(\int_{p_\star}^x \phi_1,\,\dots\,,\,\int_{p_\star}^x \phi_g \Big) \quad  \text{mod} \,\,\Lambda.$$
It is independent of the choices of path from $p_\star$ to $x$. The  Abel-Jacobi maps associated to two different points $p_\star$ and $p_\star'$ just differ by a translation on the complex torus. See   \cite{PAG} for more details.

For every $t>1$, we define $\oA_t: X^t\to \mathrm{Jac}(X)$ by 
$$\oA_t(p_1,\dots,p_t):=\sum_{j=1}^t\oA(p_j)= \sum_{j=1}^t\Big(\int_{p_\star}^{p_j} \phi_1,\,\dots \,,\,\int_{p_\star}^{p_j} \phi_g \Big) \quad  \text{mod} \,\,\Lambda.$$ 
The map $\oA_1=\oA$ is injective and  $\oA_g$ is surjective by Jacobi inversion theorem.

\subsection{Divisors on compact Riemann surfaces}\label{sec-divisor}

On  the  compact Riemann surface $X$, a \textit{divisor} is a finite formal linear combination $\oD:=\sum a_j x_j$ with $a_j\in \Z$ and $x_j\in X$. If $a_j\geq 0$ for all $j$, then we say that $\oD$ is an \textit{effective divisor}. The sum $\sum a_j$ is called the \textit{degree} of $\oD$. For a meromorphic function $f$ on $X$ (resp.\ a holomorphic  section $s$ of $L$), we write $(f)$ (resp.\ $(s)$) the divisor associated to the zeros and poles of $f$ (resp. to the zeros of $s$). 
We call \textit{principal divisor} any divisor associated to a meromorphic function. Its degree is always equal to 0. Abel's theorem says that  $\oA_t(p_1,\dots,p_t)=\oA_t(p_1', \dots ,p_t')$ if and only if  $p_1+\cdots+p_t-p_1'-\cdots-p_t'$ is a principal divisor. In this case,  we also say that $p_1+\cdots+p_t$ and $p_1'+\cdots+p_t'$ are \textit{linearly equivalent}.

We can associated to each divisor of $X$ a holomorphic line bundle, see \cite[page 134]{PAG}.
If $\mathcal D$ is a divisor on $X$, we denote by $\Oc(\mathcal D)$ the corresponding line bundle.

Denote by $\W\subset X^g$ the critical set of $\oA_g$ (this set is empty when $g=1$). 
A point  $(q_1,\dots,q_g)\in X^g$ belongs to $\W$ if and only if $\dim  H^0(X,\Oc(q_1+\cdots+q_g))\geq 2$. The image of $\W$ under $\oA_g$ is called the \textit{Wirtinger subvariety} and we denote it by $W_g^1$. It is a $2$-codimensional  analytic subset of $\mathrm{Jac}(X)$.

As a corollary of Jacobi inversion theorem, every divisor of degree greater than or equal to $g$ on $X$ is linearly equivalent to an effective divisor. For any $(p_1,\dots ,p_m)\in X^m$, there exists $(q_1,\dots,q_g)\in X^g$ such that 
\begin{equation*}\label{e:p+q=L}
\Oc(p_1+\cdots +p_m)\otimes \Oc(q_1+\cdots+q_g)\simeq L^n. 
\end{equation*}
Equivalently, if $L^n\simeq \Oc(w_1+\cdots+w_n)$, then 
\begin{equation*}\label{e:p+q=L-oA}
\oA_m(p_1,\dots ,p_m)+\oA_g(q_1,\dots,q_g)=\oA_n(w_1,\dots ,w_n).
\end{equation*}
We can define 
$$\oA_n(L^n):=\oA_n(w_1,\dots,w_n)$$
because the last quantity is independent of the choice of $w_1,\ldots,w_n$, according to Abel's theorem.
Observe that the choice of $q_1+\cdots+q_g$ is unique if $\dim H^0(X,\Oc(q_1+\cdots+q_g))=1$, or equivalently, if $\oA_n(L^n)-\oA_m(p_1,\dots ,p_m)  \notin W_g^1$. For more details on the line bundles over compact Riemann surfaces and Abel-Jacobi map, the readers may refer to \cite{demailly:agbook,PAG,gun-book}.

\begin{lemma} \label{lem-degree1}
	Let $L'$ be a positive line bundle on $X$ of degree $d > 1$. Then there exists a positive line bundle $L$ of degree 1 such that $L^d \simeq L'$.
\end{lemma}

\begin{proof}
	We consider the line bundle $L'^g$. By Jacobi inversion theorem, we can find $x_1,\dots, x_{gd}\in X$ such that $\Oc (x_1+\dots+ x_{gd})\simeq L'^g$. Let $[\mathbf z]:=\oA_{gd}(x_1,\dots ,x_{gd})$.   Using Jacobi inversion theorem again, there exist $y_1, \dots, y_g\in X$ such that $\oA_g (y_1, \dots, y_g)=[\mathbf z]/(gd)$. Abel's theorem implies that $$\Oc (gdy_1 +\cdots +gd y_g)  \otimes \Oc( -g^2 d  p_\star  )\simeq \Oc (x_1+\cdots+x_{gd}) \otimes \Oc (  - gd p_\star   ).$$
	Hence we take $L:= \Oc (y_1 +\cdots + y_g-gp_\star +p_\star)$ and finish the proof.
\end{proof}

From now on, we will assume $\deg(L)=1$. The general cases for $\deg(L)>1$ can be derived using Lemma \ref{lem-degree1}. Immediately, we have
$$N_n=n-g+1.    $$

\subsection{Admissible metric}

Given the K\"ahler form $\omega$ of integral $1$ on $X$, any line bundle $L'$ of degree one over $X$ admits a Hermitian metric $\fh'_\omega$ on $L'$ unique up to a multiplicative constant whose curvature $(1,1)$-form is exactly $\omega$.  Indeed, let $\fh'$ be any hermitian metric on $L'$ with curvature $(1,1)$-form $\omega'$ and $\int_X \omega'=1$. Since $\omega$ and $\omega'$ have the same integral, they are in the same cohomology class. So, there exists a smooth function $\varphi$ (unique up to a constant) such that $\ddc \varphi=\omega-\omega'$. It is not hard to check that the metric $\fh'_\omega=e^{-\varphi}\fh'$ satisfies the requirement. 
We say $\fh'_\omega$ is an \textit{$\omega$-admissible} hermitian metric of $L'$.

\subsection{Riemann theta divisor}
Recall that in Subsection \ref{sec-jac}, we choose the basis $\Phi$ so that $\Omega=(I,\Omega')$. Put
$$e_j:=\Big(\int_{\alpha_j} \phi_1, \, \dots \,, \,  \int_{\alpha_j} \phi_g  \Big) ,\quad \Omega_j':= \Big(\int_{\beta_j} \phi_1,\, \dots \, , \,  \int_{\beta_j} \phi_g  \Big). $$
They are the $2g$ columns of $\Omega$, which form an $\R$-basis for $\C^g$.
The complex torus $\mathrm{Jac}(X)$ admits a line bundle $L_J$ with the curvature form 
\begin{equation}\label{defn-j}
\omega_J=\sum_{j=1}^g  \dd x_j \wedge \dd x_{j+g},
\end{equation}
where $x_1,x_2,\dots,x_{2g}$ are the real coordinates  with respect to the $\R$-basis $\{e_1,\dots,e_g,\Omega_1',\dots,\Omega_g'\}$. Let $\fh_J$ be a fixed $\omega_J$-admissible hermitian metric of $L_J$.

From \cite[Page 333]{PAG}, we see that  $H^0(X,L_J)=1$. 
The line bundle $L_J$ has a global holomorphic section $\widetilde \theta$ represented by the \textit{Riemann theta function} $\theta$, which is holomorphic on $\C^g$ and satisfying
$$\theta(\mathbf z+e_j)=\theta(\mathbf z) \quad \text{and}\quad \theta(\mathbf z+\Omega'_j)=e^{-2\pi i(\Omega'_j+\Omega'_{j,j}/2)}\theta(\mathbf z)$$
for every $\mathbf z\in \C^g$ and $1\leq j\leq g$. Here, $\Omega_{j,j}'$ is  the $(j,j)$-entry of the symmetric matrix $\Omega'$. 
Let $\Theta:=(\widetilde \theta)$ be the divisor of $\widetilde\theta$, called the \textit{Riemann theta divisor}.
Riemann's  theorem (c.f.\ \cite[Page 338]{PAG}) states that there exists a unique point $z_\star\in \mathrm{Jac}(X)$ such that $W_{g-1}=\Theta +z_\star$. This $z_\star$ depends on the choice of $p_\star$ and $\widetilde \theta$.

\subsection{Green function}

For every $x\in X$, consider the mean normalized $\omega$-potential $U^\star_{\delta_x}$ of $\delta_x$. Gluing all such functions together, we get the \textit{Green function} defined on $X\times X$ (c.f.\ \cite[Theorem 2.2]{Lang}).

\begin{lemma} \label{l:Green}
	There exists a Green function $G(x, y)$ of $(X, \omega)$ defined on $X \times X$, which is smooth outside the diagonal, such that for every $x \in X$,
	\[
	\int_X G(x, \cdot) \, \omega = 0 \qquad \text{and} \qquad \ddc G(x, \cdot) = \delta_x - \omega.
	\]
	Moreover, the function $\varrho(x, y) := G(x, y) - \log \dist(x, y)$ is Lipschitz.
\end{lemma}

For every probability measure $\nu$, it is not hard to check that
\begin{equation}\label{defn-potentil-type-I}
U^\star_\nu(x) = \int_X G(x, \cdot) \, \dd\nu .
\end{equation}
In particular,  $U^\star_\mu$ is uniformly bounded from above.

One can also define the mean normalized $\omega$-potential $U^\star_\nu$ for any positive or signed measure $\nu$ on $X$ using  \eqref{defn-potentil-type-I}. For any two positive or signed measures $\nu_1, \nu_2$ on $X$, observe that $U^\star_{\nu_1 + \nu_2} = U^\star_{\nu_1} + U^\star_{\nu_2}$. Moreover, by Stokes' formula, one can show that
\begin{equation}
\label{commute-potential}
\int_X U^\star_{\nu_1} \, \dd\nu_2 = \int_X U^\star_{\nu_2} \, \dd\nu_1,
\end{equation}
and
\begin{equation}
\label{positive energy}
\int_X U^\star_{\nu_1} \, \dd\nu_1 \leq 0.
\end{equation}

For point mass measure $\delta_x$,
there is another way to describe its mean normalized $\omega$-potential $U^\star_{\delta_x}$ by considering the line bundle $\Oc(x)$. Recall that every holomorphic section of $\Oc(x)$ corresponds to a meromorphic function  $f$ such that $(f)+x$ is an effective divisor of degree one. Denote by $\mathbf 1_{\Oc(x)}$ the holomorphic section of $\Oc(x)$  corresponding to the  function $\mathbf 1$. Let $\fh_\omega$ be an $\omega$-admissible metric on $\Oc(x)$. Then it is not hard to see that 
$ \ddc \log |\mathbf 1_{\Oc(x)}|_{\fh_\omega}=\delta_x -\omega$.
So, $$U^\star_{\delta_x}(z)=\log  \big|\mathbf 1_{\Oc(x)}(z)\big|_{\fh_\omega}+c_{\fh_\omega}$$
 for some constant  $c_{\fh_\omega}$ depending on the admissible metric chosen.

For every point $\mathbf p:=(p_1,\dots,p_m)\in X^m,m\in \N$, we denote by $\delta_{\mathbf p}$ the probability measure $(\delta_{p_1}+\cdots+\delta_{p_m})/m$. Here we omit the dependence on $m$ to ease the notation. Similarly as above, we have
\begin{equation}\label{bundlepotential} 
U^\star_{\delta_{\mathbf p}}(z)={1\over m}\sum_{1\leq j\leq m}  U^\star_{\delta_{p_j}}(z) = {1\over m}\sum_{1\leq j\leq m} \Big(\log \big|\mathbf 1_{\Oc(p_j)}(z)\big|_{\fh_j}+c_{\fh_j}\Big),
 \end{equation}
where $\fh_j$ is an $\omega$-admissible metric on $\Oc(p_j)$.
Moreover,   we can take the same weight function $\varphi$ for all the $\fh_j$'s. More precisely, in every small open ball $B$, there exist holomorphic functions $g_j$'s on $B$, such that for all $j$,
\begin{equation}\label{sameweight} |\mathbf 1_{\Oc(p_j)}|_{\fh_j} =| g_j| e^{-\varphi}   \quad\text{on }\, B.   \end{equation}
Indeed,  on the simply connected set $B$, every harmonic function is the real part of a holomorphic function.

\subsection{Upper envelop}

For a strictly negative continuous function $u$ on $\partial K$, define the upper envelop
$$  \widehat U:=    \sup_\varphi  \big\{  \varphi\in\SH(X,\omega): \; \varphi\leq 0 \text{ on } X,\; \varphi \leq u \text{ on }  \partial K  \big\}.$$    
We have the following  result, see e.g., \cite[Lemma 2.3]{wu-hole}. 
\begin{lemma}  \label{lem-upenvop}
	The function $\widehat U$ is a continuous $\omega$-subharmonic functions on $X$ satisfying
	$$ \ddc \widehat U=-\omega   \,\text{ on }\, \{\widehat U\neq 0 \}\setminus \partial D   \quad\text{and}\quad  \widehat U=u  \,\text{ on }\, \partial D.         $$
\end{lemma}

\subsection{Discrete energy}

For every point $\mathbf p:=(p_1,\dots,p_m)\in X^m,m\in \N$, we define the \textit{discrete energy} of $\mathbf p$ as
$$\oE_m(\mathbf p):=\int_{(X\times X)\setminus \Delta}G(x,y)\,\dd \delta_{\mathbf p}(x)\dd \delta_{\mathbf p}(y)={1\over m^2} \sum_{j\neq k} G(p_j,p_k),$$
Here, $G$ is the Green function defined in Lemma \ref{l:Green}.  Obviously, if $p_j=p_k$ for some $j\neq k$,  then $\oE_m(\mathbf p )=-\infty$.

\section{Bernstein-Markov condition and the equilibrium measure}

\subsection{Bernstein-Markov condition}\label{sec-BM}

Define the \textit{Bergman function} of $L^n$, associated with the weighted subset $(K,\phi,\mu)$ as
$$\rho_n(\mu,\phi)(x):=\sup \big\{ |s(x)|_{n\phi} :\, s\in H^0(X,L^n),\, \norm{s}_{L^2(\mu,n\phi)}=1 \big\}.$$
Equivalently,
$$ \rho_n(\mu,\phi)(x):=  \sum_{1\leq j\leq N_n} |s_j(x)|^2_{n\phi} $$
for any orthonormal  basis   $(s_1,s_2,\dots,s_{N_n})$  of $H^0 (X,L^n)$ with respect to $L^2(\mu,n\phi)$-norm.

We say that  the weighted subset $(K,\phi,\mu)$ satisfies the \textit{Bernstein-Markov condition} if for every $\vep>0$, there exists a constant $c_\vep>0$ such that 
$$\sup_K \rho_n(\mu,\phi)\leq c_\vep e^{\vep n}  \quad\text{for all }\, n>1.$$

We say that  the weighted subset $(K,\phi,\mu)$ satisfies the \textit{strong Bernstein-Markov condition} if there exists a constant $C>0$ such that 
$$\sup_K \rho_n(\mu,\phi)\leq Cn^C  \quad\text{for all }\, n>1.$$

In this article, we will only focus on the strong Bernstein-Markov condition.
Recall the following two results from \cite[Lemma 3.5, Lemma 3.6]{din-ma-ngu-ens}.

\begin{lemma}\label{lem-dmn}
	If the weighted subset $(K,\phi,\mu)$ satisfies the strong Bernstein-Markov condition, then there exists a constant $C>0$ independent of $n$, such that 
	$$ 0\leq \oL_n(\mu,\phi) -\oL_n(K,\phi)\leq C n^{-1}\log n \quad\text{for all }\, n>1.  $$
\end{lemma}

\begin{lemma}\label{lem-strongBM}
	The weighted subset $(X,\phi,\omega)$ for $\phi\in\Cc^\gamma(X)$ satisfies the strong Bernstein-Markov condition.
\end{lemma}

\subsection{Minimizer of the functional $\oI_\phi$} \label{sec-min}

In the following, we will give the solution of the problem
$$  \min \oI_\phi(\nu), \quad  \supp(\nu)\subset \cM(K).           $$

Define 
$$U_{K,\phi}:= \sup_\varphi  \big\{  \varphi\in \SH(X,\omega): \; \varphi \leq \phi \,\text{ on }\, K\big\}.$$
The following regularity of $U_{K,\phi}$ was given in \cite[Theorem 2.7]{din-ma-ngu-ens}.

\begin{lemma}\label{lem-UK}
If   $K$ is the closure of a non-empty open subset of $X$ with $\Cc^2$ boundary and  $\phi$ is  $\gamma$-H\"older, then	the function $U_{K,\phi}$ is  $\gamma$-H\"older and $\norm{U_{K,\phi}}_{\Cc^\gamma}\leq c\norm{\phi}_{\Cc^\gamma}$ for some $c>0$ independent of $\phi$.    In particular, $U_{K,\phi}$ is $\omega$-subharmonic and $\ddc U_{K,\phi}=-\omega$ on $X\setminus \{ U_{K,\phi}=\phi\}$ and $X\setminus K$.
\end{lemma}

Thus, we can define the probability measure $$\nu_{K,\phi}:=\ddc U_{K,\phi}+\omega.$$ By Lemma \ref{lem-UK}, the support of $\nu_{K,\phi}$ is contained in $K$.

\begin{proposition}\label{sol-min}
	The minimizer of $\oI_\phi$ over $\cM(K)$ is the probability measure $\nu_{K,\phi}$. Moreover, one has
	$$ U^\star_{\nu_{K,\phi}}=\phi +\int_X \phi \,\dd\nu_{K,\phi}-\min_{\cM(K)}\oI_\phi \quad\text{on }\, \supp(\nu_{K,\phi}) $$
	and 
	$$U^\star_{\nu_{K,\phi}}=U_{K,\phi} +\int_X \phi \,\dd\nu_{K,\phi}-\min_{\cM(K)}\oI_\phi \quad\text{on }\, X.$$
\end{proposition}

\begin{proof}
	There is no difficulty to adapt the proof of \cite[Page 27, Theorem 1.3]{log-book}, although the setting there is on $\C$. We only need to verify the identities in the lemma.  Note that $U^\star_{\nu_{K,\phi}}-U_{K,\phi}=c$ for some constant $c$. Using that $U_{K,\phi}=\phi$ on $\supp(\nu_{K,\phi})$, we get
	\begin{align*}
	\min_{\cM(K)}\oI_\phi&=\oI_\phi (\nu_{K,\phi})=-\int_X U^\star_{\nu_{K,\phi}} \,\dd\nu_{K,\phi} +2\int_X \phi \,\dd\nu_{K,\phi}  \\
	&=-\int_X U_{K,\phi} \,\dd\nu_{K,\phi} -c+2\int_X \phi \,\dd\nu_{K,\phi}\\
	&=-\int_X \phi \,\dd\nu_{K,\phi} -c+2\int_X \phi \,\dd\nu_{K,\phi}=\int_X \phi \,\dd\nu_{K,\phi}-c.
	\end{align*}
	
	Thus, on $\supp(\nu_{K,\phi})$, we have 
	$$U^\star_{\nu_{K,\phi}}=\phi+c =   \phi +\int_X \phi \,\dd\nu_{K,\phi}-\min_{\cM(K)}\oI_\phi,     $$
   giving the first equation. For the second one, we only need to use $U_{K,\phi}=\phi$ on $\supp(\nu_{K,\phi})$ one more time.
\end{proof}

The following result allows us to replace the weighted subset $(K,\phi,\mu)$ by $(X,U_{K,\phi},\mu)$ for the unit balls with respect to  $L^\infty (K,n\phi)$-norm   in $H^0(X,L^n)$.

\begin{lemma}\label{lem-K-X}
	For every section $s\in H^0(X,L^n)$, one has $\norm{s}_{L^\infty (K,n\phi)}=\norm{s}_{L^\infty (X,n U_{K,\phi})}$. In particular,
	$$ \oL_n(K,\phi)=\oL_n(X,U_{K,\phi}).   $$
\end{lemma}

\begin{proof}
	We only need to show the first assertion, as the second one is a consequence by definition of $\oL_n(K,\phi)$.   Consider the function    
	$$\Phi(x):= \log |s(x)|_{n U_{K,\phi}}=\log|s(x)|_h-nU_{K,\phi}(x).$$
	Note that $\ddc |s|_h\geq -n\omega$ on $X$.
	 Lemma \ref{lem-UK} gives $\ddc U_{K,\phi}=-\omega$ on $X\setminus \{U_{K,\phi}=\phi\}$ and $X\setminus K$. Thus, $\ddc \Phi \geq 0$ on  $X\setminus \{U_{K,\phi}=\phi\}$ and $X\setminus K$. By maximal modulus principle, the maximum of $\Phi$ appears on  $\{U_{K,\phi}=\phi\}\cap K$, which implies the maximum of $|s(x)|_{n U_{K,\phi}}$ appears on  $\{U_{K,\phi}=\phi\}\cap K$. We conclude that $\norm{s}_{L^\infty (X,n U_{K,\phi})}\leq \norm{s}_{L^\infty (K,n\phi)}$.
	 
	 On the other hand, by definition, $U_{K,\phi}\leq \phi$ on $K$. This implies $|s(x)|_{n \phi}\leq |s(x)|_{n U_{K,\phi}}$ for $x\in K$ and hence $\norm{s}_{L^\infty (K,n\phi)}\leq \norm{s}_{L^\infty (X,n U_{K,\phi})}$. The proof of the lemma is finished.
\end{proof}

The measure $\nu_{K,\phi}$ is called the \textit{equilibrium measure} of $(K,\phi,\mu)$. From now on until the end of the paper, we only consider the case when $K$ is the closure of a non-empty open subset with $\Cc^2$ boundary and $\phi$  is     $\gamma$-H\"older continuous.

\section{Bosonization formula}

To prove Theorems \ref{maintheorem1}, \ref{maintheorem2} and \ref{maintheorem3}, it is enough to fix a special  $(K_2,\phi_2,\mu_2)$, and use the triangular inequality to conclude. The simplest choice is  $$(K_2,\phi_2,\mu_2):=(X,\mathbf{0},\omega ),$$
which clearly satisfy the strong Bernstein-Markov condition and the mass-density condition.  Moreover, by \eqref{positive energy}, 
\begin{equation}\label{minI0}
    \min_{\cM(X)} \oI_{\mathbf 0} =   \min_{\nu\in\cM(X)} -\int_X U^\star_\nu \,\dd \nu  =-\int_X U^\star_{\omega } \,  \omega =-\int_X 0 \,  \omega =0.      \end{equation}

To simplify the notation, 
we
write $(K_1,\phi_1,\mu_1)$ as $(K,\phi,\mu)$. Our goal is to estimate the two differences 
$$  \oL_n(K,\phi)-\oL_n(X,\mathbf{0}) \quad\text{and}\quad   \oL_n(\mu,\phi)-\oL_n(\omega ,\mathbf{0}).$$

Recall the following identity in \cite[Lemma 5.3]{ber-bou-invent}.

\begin{lemma}\label{l2formula}
	If $S_n:=(s_1,s_2,\dots,s_{N_n})$ is a basis of $H^0 (X,L^n)$, then
	$$ \norm{\det S_n}_{L^2(\mu,n\phi)}^2 =N_n ! \det \big( \lp  s_j,s_k \rp_{L^2(\mu,n\phi)}  \big)_{j,k}.  $$
\end{lemma}

In particular, when $S_n$ is orthonormal, the right hand side in the equation above is $N_n!$.
From now on until the end of the paper, we fix $S_n$ as an $L^2(\omega ,n\mathbf{0})$-orthonormal basis of $H^0(X,L^n)$.
By Lemma \ref{l2formula},
\begin{equation} \label{key-equation}
\oL_n(\omega ,\mathbf{0})-\oL_n(\mu,\phi)={1\over n N_n}\log  \norm{\det S_n}_{L^2(\mu,n\phi)}^2 -{1\over n N_n} \log(N_n !),
\end{equation}
We will use the following \textit{bosonization formula} to estimate $\norm{\det S_n}_{L^2(\mu,n\phi)}^2$.

\begin{lemma}\label{boson}
	For all $(x_1,\dots,x_{N_n})\in X^{N_n}$,
	$$
	\big|\det S_n(x_1,\dots,x_{N_n})\big|_h^2 =\mathcal Z_n e^{\sum_{1\leq \ell\leq N_n}\lambda(x_\ell)} e^{\sum_{j\neq k} G(x_j,x_k)}   \Big\|\widetilde\theta\big(  \oA_{n}(L^{n})-\sum_{1\leq j\leq N_n}\oA_1(x_j) -\mathbf z_\star \big)\Big\|_{\fh_J}^2, 
	$$
	where $\mathcal Z_n>0$ is constant independent of $x_1,\dots,x_{N_n}$ and $\lambda$ is a smooth function   independent of $n,x_1,\dots,x_{N_n}$.
\end{lemma}

\begin{proof}
	By \cite[equation (3.27)]{kle-cmp} (see also \cite[equation (5.4)]{MR0908551} and \cite[Theorem 5]{zel-imrn}), Lemma \ref{boson} holds for the \textit{Arakelov measure} $\omega_A$ with $\lambda=\mathbf 0$. More precisely,
	\begin{equation}\label{boark}
	\big|\det S_n(x_1,\dots,x_{N_n})\big|_{\fh_{A}}^2 =\mathcal Z'_n  e^{\sum_{j\neq k} G_A(x_j,x_k)}   \Big\|\widetilde\theta\big(  \oA_{n}(L^{n})-\sum_{1\leq j\leq N_n}\oA_1(x_j) -\mathbf z_\star \big)\Big\|_{\fh_J}^2,
	\end{equation}
	where $\fh_A$ is the \textit{Arakelov metric} induced by $\omega_A$, $\mathcal Z'_n>0$ is  a constant and $G_A$ is the   Green function on $X\times X$ associated to $\omega_A$. The Arakelov measure $\omega_A$ is the smooth volume form on $X$, defined by the pull-back of the $(1,1)$-form $\omega_J$ \eqref{defn-j} on $\mathrm{Jac}(X)$ by $\oA_1$ after normalization. For the proof here, we only need to use that $\omega_A$ is a   smooth volume form of integral $1$.
	
	Consider  the function
	$$\varphi(x):=-\int_X  G(x,y)\,\omega_A(y)+{1\over 2}\int_{X\times X} G(z,y)\, \omega_A(z)\omega_A(y).$$
	It is not hard to see that $\varphi$ is smooth, 
	\begin{equation}\label{arkgreen}
	\ddc \varphi = \omega-\omega_A \quad\text{and}\quad G_A(x,y)=G(x,y)+\varphi(x)+\varphi(y)  
	\end{equation}
	by checking the two conditions in Lemma \ref{l:Green} for the Green function.
	
	On the other hand,  for each $s_j$, since it is a section of $L^n$, by the first equation in \eqref{arkgreen}, we have
	$$ |s_j(x)|_h =c_\varphi e^{-n\varphi(x)} |s_j(x)|_{\fh_A}  \quad\text{for some constant }\, c_\varphi>0.$$
	By definition, $\det S_n$ is a linear combination of the $N_n$ terms product  $\prod s_j (x_\ell)$, where  each $s_j$ and $x_\ell$ appear exactly once. It follows that 
	\begin{equation}\label{arksection}
	\big|\det S_n(x_1,\dots,x_{N_n})\big|_{h}=c_\varphi^{N_n} e^{-n\sum_{1\leq \ell\leq N_n}\varphi(x_\ell)}\big|\det S_n(x_1,\dots,x_{N_n})\big|_{\fh_{A}}.
	\end{equation}
	
	Now we substitute \eqref{arkgreen} and \eqref{arksection} into \eqref{boark}, getting
	\begin{align*}
	&\big|\det S_n(x_1,\dots,x_{N_n})\big|_h^2 \\
	&= c_\varphi^{2N_n}\mathcal Z'_n  e^{\sum_{j\neq k} G_A(x_j,x_k)-2n\sum_{1\leq \ell\leq N_n}\varphi(x_\ell)}\Big\|\widetilde\theta\big(  \oA_{n}(L^{n})-\sum_{1\leq j\leq N_n}\oA_1(x_j) -\mathbf z_\star \big)\Big\|_{\fh_J}^2\\
	&= c_\varphi^{2N_n}\mathcal Z'_n  e^{\sum_{j\neq k} G(x_j,x_k)-2\sum_{1\leq \ell\leq N_n}\varphi(x_\ell)}\Big\|\widetilde\theta\big(  \oA_{n}(L^{n})-\sum_{1\leq j\leq N_n}\oA_1(x_j) -\mathbf z_\star \big)\Big\|_{\fh_J}^2.
	\end{align*}
	We finish the proof of the proposition by letting $\mathcal Z_n:=c_\varphi^{2N_n}\mathcal Z'_n$ and $\lambda:=-2\varphi$.
\end{proof}

By Lemma \ref{boson} and definition, the  quantity $\norm{\det S_n}_{L^2(\mu,n\phi)}^2$ in \eqref{key-equation} can be written as
\begin{align*}
\norm{\det S_n}_{L^2(\mu,n\phi)}^2 =\int_{K^{N_n}} &\mathcal Z_n e^{\sum_{1\leq \ell\leq N_n}\lambda(x_\ell)} e^{\sum_{j\neq k} G(x_j,x_k)-2n\sum_{1\leq j\leq N_n} \phi(x_j)}  \\
& \Big\|\widetilde\theta\big(  \oA_{n}(L^{n})-\sum_{1\leq j\leq N_n}\oA_1(x_j) -\mathbf z_\star \big)\Big\|_{\fh_J}^2 \,  \dd \mu^{N_n}(x_1,\dots,x_{N_n}).   
\end{align*}
The dominated term inside the integral is
\begin{equation}\label{e-e}
e^{\sum_{j\neq k} G(x_j,x_k)-2n\sum_{1\leq j\leq N_n} \phi(x_j)}  = \exp \left[ N_n^2  \oE_{N_n}(\mathbf x) - 2nN_n\int_X \phi \,\dd \delta_{\mathbf x}      \right]    
\end{equation}
where $\mathbf x:=(x_1,\dots,x_{N_n})$.  The main ingredient of this article is to estimate last quantity. For this aim, we consider the following crucial \textit{discrete energy functional}
$$ \oJ_{m,\phi} (\mathbf p) := -\oE_m(\mathbf p) +2\int _X  \phi \,\dd\delta_{\mathbf p}, \quad  \mathbf p\in X^m,\quad m\in\N.         $$ 
Observe that
\begin{equation*} \label{diff-e-j}
\Big|N_n^2  \oE_{N_n}(\mathbf x) - 2nN_n\int_X \phi \,\dd \delta_{\mathbf x} + N_n^2 \oJ_{N_n,\phi} (\mathbf x)  \Big| = 2N_n(g-1) \Big| \int_X \phi \,\dd \delta_{\mathbf x}\Big| = O(N_n).  
\end{equation*}
Thus, to estimate \eqref{e-e}, we need to analyze the functional $\oJ_{N_n,\phi}$.

\section{Lower bound of $\oJ_{m,\phi}$}\label{sec-com-1}

The new defining functional $\oJ_{m,\phi}$ can be seen as the discrete version of $\oI_\phi$. We will show that their minimal values only differ  by an error  $O(m^{-1}\log m)$. In this section, we give a lower bound of $\oJ_{m,\phi}$.

\begin{proposition}\label{prop-upper-functional}
	There exists a constant $C>0$ independent of $\phi$ and $m$, such that
	$$\inf_{\mathbf p\in K^m}  \oJ_{m,\phi}  \geq  \min_{\cM(K)} \oI_\phi  -C(1+\norm{\phi}_{\Cc^\gamma}) {\log m\over m}  \quad\text{for all}\quad m>1.$$
\end{proposition}

We need to show that for all $\mathbf{p} =( p_1 ,\dots , p_m )\in K^m$, 
\begin{equation}\label{ineq-mu-p}
\oE_m(\mathbf{p}) -  2\int _X  \phi \,\dd\delta_{\mathbf p} \leq - \min_{\cM(K)} \oI_\phi + C(1+\norm{\phi}_{\Cc^\gamma})  \frac{\log m}{m}.
\end{equation}
This immediately implies Proposition~\ref{prop-upper-functional}. Without loss of generality, we may assume that the points $p_1, p_2, \dots, p_m$ are all distinct, as otherwise $\oE_m(\mathbf{p}) =-\infty$ and the inequality holds trivially. Under this assumption, $U^\star_{\delta_{p_j}}(p_k)$ is finite for $j\neq k$.

\medskip

For each $1\leq j\leq m$, consider the function
\[
V_{p_j} := \sup_\varphi  \big\{  \varphi\in \SH(X,\omega): \; \varphi \leq U^\star_{\delta_{p_j}} \,\text{ on }\, \partial  \B(p_j, m^{-1/\gamma}) \big\}.
\]
By Lemma \ref{lem-upenvop} (apply that lemma with $u=U^\star_{\delta_{p_j}}-C$ for large $C$), $V_{p_j}$ is a continuous $\omega$-subharmonic function satisfying:
$$
V_{p_j}  = U^\star_{\delta_{p_j}} \, \text{ on }\, X \setminus \B(p_j, m^{-1/\gamma}) \quad\text{and}\quad
V_{p_j}  > U^\star_{\delta_{p_j}} \, \text{ on }\, \B(p_j, m^{-1/\gamma}), $$ 
$$
\ddc V_{p_j} = -\omega \,\text{ on }\, X \setminus \partial \B(p_j, m^{-1/\gamma}).
$$
Define the probability measure $$\sigma_{p_j} := \ddc V_{p_j} + \omega,$$ which is supported on $\partial \B(p_j, m^{-1/\gamma})$. Finally, we define their average 
$$    \sigma_{\mathbf p}:={1\over m}\sum_{1\leq j\leq m} \sigma _{p_j} .    $$

\begin{lemma}\label{lem-delta-sigma}
	One has  $|V_{p_j} -U^\star_{\sigma_{p_j}} |\leq C m^{-2}\log m$  for some $C>0$ independent of $m$ and ${p_j}$. In particular, $|U^\star_{\sigma_{p_j}}- U^\star_{\delta_{p_j}}|\leq Cm^{-2}\log m$ on $X\setminus \B(p_j, m^{-1/\gamma})$.
\end{lemma}

\begin{proof}
	By the definition of the mean normalized $\omega$-potential, we have $V_{p_j} - U^\star_{\sigma_{p_j}} = \int_X V_{p_j} \, \omega$. Using the fact that $V_{p_j} = U^\star_{\delta_{p_j}}$ on $X \setminus \B(p_j, m^{-1/\gamma})$, we obtain
	\begin{align*}
	\int_X V_{p_j} \, \omega &= \int_{\B(p_j, m^{-1/\gamma})} V_{p_j} \, \omega + \int_{X \setminus \B(p_j, m^{-1/\gamma})} V_{p_j} \, \omega \\
	&= \int_{\B(p_j, m^{-1/\gamma})} V_{p_j} \, \omega + \int_{X \setminus \B(p_j, m^{-1/\gamma})} U^\star_{\delta_{p_j}} \, \omega \\
	&=   \int_{\B(p_j, m^{-1/\gamma})} V_{p_j} \, \omega - \int_{\B(p_j, m^{-1/\gamma})} U^\star_{\delta_{p_j}} \, \omega.
	\end{align*}
	Recall that $U^\star_{\delta_{p_j}}$ is uniformly bounded from above, independent of $p_j$. On $\B(p_j, m^{-1/\gamma})$, the function $V_{p_j}$ satisfies $\ddc V_{p_j} = -\omega$ with boundary condition $V_{p_j} = U^\star_{\delta_{p_j}}$ on $\partial \B(p_j, m^{-1/\gamma})$, which implies $V_{p_j} \leq c_1$ for some constant $c_1 > 0$ independent of $p_j$. Moreover, from formula \eqref{defn-potentil-type-I} and Lemma \ref{l:Green}, we have for $z \in \overline\B(p_j, m^{-1/\gamma})$,
	\begin{equation}\label{lowerbound-Udelta}
	U^\star_{\delta_{p_j}}(z) \geq c_2 \log \dist(z, p_j) ,
	\end{equation}
	where $c_2 > 0$ is also independent of $p_j$.
	Since $V_{p_j} > U^\star_{\delta_{p_j}}$ on $\B(p_j, m^{-1/\gamma})$, we have for $z\in X$,
	\begin{align*}
	\big| V_{p_j}(z) - U^\star_{\sigma_{p_j}}(z) \big| &= \Big| \int_X V_{p_j} \, \omega \Big| \leq \Big| \int_{\B(p_j, m^{-1/\gamma})} V_{p_j} \, \omega \Big| + \Big| \int_{\B(p_j, m^{-1/\gamma})} U^\star_{\delta_{p_j}} \, \omega \Big| \\
	&\leq \int_{\B(p_j, m^{-1/\gamma})} \max\big( c_1, |U^\star_{\delta_{p_j}}| \big) \, \omega + \int_{\B(p_j, m^{-1/\gamma})} |U^\star_{\delta_{p_j}}| \, \omega \\
	&\leq c_1 \cdot \Area_{\omega}\big(\B(p_j, m^{-1/\gamma})\big) + 2 \int_{\B(p_j, m^{-1/\gamma})} |U^\star_{\delta_{p_j}}| \, \omega.
	\end{align*}
	The first term is $O(m^{-2/\gamma})$.
	We now estimate the second term. Recall the estimate
	\[
	\Big| \int_{\mathbb{D}(0, r)} \log |z|  \, i  \dd z \wedge \dd\bar{z} \Big| = O(r^2 |\log r|).
	\]
	Taking \eqref{lowerbound-Udelta} into account, the second term  is also $ O(m^{-2/\gamma} \log m) $. This completes the proof of the lemma since $\gamma\leq 1$.
\end{proof}

We are now going to prove a crucial inequality linking the discrete energy functional $\oJ_{m,\phi}$ and the energy functional $\oI_\phi$.

\begin{lemma}\label{prop-find-p-1}
	There exists a constant $C>0$ independent of $\phi,m$ and $\mathbf p$ such that for all $m>1$,
$$
	\oE_m(\mathbf p)\leq \int_X  U^\star_{\sigma_{\mathbf p}}\,\dd \sigma_{\mathbf p} +C {\log m \over m}
$$
	and
	$$  \Big| \int_X \phi \,\dd \delta_{\mathbf p}-\int_X \phi \,\dd \sigma_{\mathbf p}\Big| \leq   \norm{\phi}_{\Cc^\gamma}  m^{-1}.    $$
\end{lemma}

\begin{proof} 
	We use formula \eqref{defn-potentil-type-I} to rewrite
	\begin{equation}\label{exp-oE}
	\oE_m(\mathbf p) = \frac{1}{m^2} \sum_{j \neq k} G(p_j,p_k)= \frac{1}{m^2} \sum_{j \neq k} U^\star_{\delta_{p_j}}(p_k).
	\end{equation}
	All terms are finite in the summation.
   For each $j$,  Lemma \ref{lem-delta-sigma} gives
   $$ U^\star_{\delta_{p_j}}\leq V_{p_j} \leq   U^\star_{\sigma_{p_j}}+O(m^2\log m). $$
   For each $j,k$, by \eqref{commute-potential}, we have
   $$ U^\star_{\sigma_{p_j}} (p_k)=\int_X  U^\star_{\sigma_{p_j}} \,\dd \delta_{p_k} =\int_X U^\star_{\delta_{p_k}} \,\dd \sigma_{p_j}\leq \int_X U^\star_{\sigma_{p_k}} \,\dd \sigma_{p_j}+O(m^2\log m).$$  
We conclude from last two inequalities that
$$\oE_m(\mathbf p)\leq {1\over m^2}  \sum_{j \neq k}  \int_X U^\star_{\sigma_{p_k}} \,\dd \sigma_{p_j}+O(m^2\log m).   $$

By definition of $\sigma_{\mathbf p}$,
$$\int_X  U^\star_{\sigma_{\mathbf p}}\,\dd \sigma_{\mathbf p}={1\over m^2}\sum_{j,k} \int_X U^\star_{\sigma_{p_k}} \,\dd \sigma_{p_j}. $$
Thus, to prove the first assertion, it remains to show 
\begin{equation}\label{pkpk}
\Big|{1\over m^2} \sum_{1\leq k\leq m} \int_X U^\star_{\sigma_{p_k}} \,\dd \sigma_{p_k}\Big| =O\Big( {\log m\over m}  \Big)
\end{equation}    

  Lemma \ref{lem-delta-sigma} says that $|U^\star_{\sigma_{p_k}}-U^\star_{\delta_{p_k}}|=O( m^{-2}\log m)$ on the support of $\sigma_{p_k}$. Combining with \eqref{lowerbound-Udelta} and the fact that $U^\star_{\delta_{p_k}}$ is uniformly bounded from above, we conclude that 
  $$\Big|\int_X U^\star_{\sigma_{p_k}} \sigma_{p_k}\Big|=\Big|\int_X U^\star_{\delta_{p_k}} \sigma_{p_k}\Big|+O(m^{-2}\log m)\leq O(\log m). $$
This implies \eqref{pkpk} and gives the first assertion.
	
	\smallskip
	
	For the second assertion, recall that $\phi$ is $\gamma$-H\"older. Using that the distance from every point in $\supp(\sigma_{p_j})$  to $p_j$ is $m^{-1/\gamma}$, we get 
	$$\Big|\int_X \phi \,\dd \delta_{p_j}-\int_X \phi\,\dd \sigma_{p_j}\Big|\leq \norm{\phi}_{\Cc^\gamma}( m^{-1/\gamma})^\gamma= \norm{\phi}_{\Cc^\gamma} m^{-1}$$
		Therefore,
	$$ \Big| \int_X \phi \,\dd \delta_{\mathbf p}-\int_X \phi \,\dd \sigma_{\mathbf p}\Big| \leq  {1\over m}\sum_{1\leq j\leq m} \Big| \int_X \phi \,\dd \delta_{p_j}-\int_X \phi\,\dd \sigma_{p_j} \Big|   \leq \norm{\phi}_{\Cc^\gamma} m^{-1}. $$
This gives the second assertion and finishes the proof of the lemma.
\end{proof}

\begin{remark}\rm
In the first inequality in Lemma \ref{prop-find-p-1}, there is no lower bound. This is because when $\dist(p_j,p_k)$  is very small, $\oE_m(\mathbf p)$ will be very negative.   When we assume the points $p_j$'s are  far away from each other, then we also have the lower bound, see Lemma \ref{lem-oE-sigma} later.
\end{remark}

If $p_j$ is too close to the boundary of $K$, the support of $\sigma_{\mathbf p}$  may exceed $K$.
By Lemma \ref{prop-find-p-1}, we get the following inequality quite close to \eqref{ineq-mu-p}:
$$\oE_m(\mathbf{p}) -  2\int _X  \phi \,\dd\delta_{\mathbf p} \leq - \min_{\cM(K_m)} \oI_\phi + C(1+\norm{\phi}_{\Cc^\gamma})  \frac{\log m}{m}.$$
where
$$K_m:=\big\{x\in X:\,  \dist(x,K)\leq m^{-1/\gamma}   \big\}$$
is  the closure of $m^{-1/\gamma}$-neighborhood of $K$.
 We need to do one more perturbation to get Proposition \ref{prop-upper-functional}.

\begin{lemma}\label{lem-Km}
	We have
	$$  \big|\min_{\cM(K)} \oI_\phi - \min_{\cM(K_m)} \oI_\phi \big|\leq 6(\norm{U_{K,\phi}}_{\Cc^\gamma}+\norm{\phi}_{\Cc^\gamma})m^{-1}\leq C\norm{\phi}_{\Cc^\gamma}m^{-1}        $$
	for some $C>0$ independent of $\phi$ and $m$ .
\end{lemma}

\begin{proof}
	Put $A:=\norm{U_{K,\phi}}_{\Cc^\gamma}+\norm{\phi}_{\Cc^\gamma}$.
	Denote by  $\nu_{K_m}$  the minimizer of $\oI_\phi$ over $\cM(K_m)$.
	Recall from Proposition \ref{sol-min} that $\nu_{K,\phi}=\ddc U_{K,\phi}+\omega,\nu_{K_m}=\ddc U_{K_m} +\omega$, where
	$$U_{K,\phi}:= \sup_\varphi  \big\{  \varphi\in \SH(X,\omega): \; \varphi \leq \phi \,\text{ on }\, K\big\};$$
	$$U_{K_m}:= \sup_\varphi  \big\{  \varphi\in \SH(X,\omega): \; \varphi \leq \phi \,\text{ on }\, K_m\big\}.$$
	Obviously, $U_{K,\phi}\geq U_{K_m}$ since $K\subset K_m$.  Moreover,  $U_{K,\phi}\leq \phi$ on $K$   implies 
	  $$U_{K,\phi}-\phi\leq (\norm{U_{K,\phi}}_{\Cc^\gamma}+\norm{\phi}_{\Cc^\gamma})\cdot (m^{-1/\gamma})^\gamma =A m^{-1} \quad\text{on }\,  K_m.$$  
	   This means that $U_{K,\phi}-Am^{-1}\leq U_{K_m}$ from the definition of $U_{K_m}$, and hence
	   $$ \int_X U_{K,\phi}\,\omega -Am^{-1}\leq  \int_X U_{K_m}\,\omega \leq \int_X U_{K,\phi}\,\omega. $$
	    Thus, by the definition of mean normalized $\omega$-potential, we conclude that
	    $$|U^\star_{\nu_{K,\phi}}-U^\star_{\nu_{K_m}}|\leq Am^{-1}.$$
	Therefore, applying \eqref{commute-potential} several times, we get
	\begin{align*}
	&\Big| \int_X  U^\star_{\nu_{K,\phi}}  \,\dd  \nu_{K,\phi}- \int_X U^\star_{\nu_{K_m}} \,\dd \nu_{K_m}     \Big| \leq \Big| \int_X  U^\star_{\nu_{K,\phi}}  \,\dd  \nu_{K,\phi}- \int_X U^\star_{\nu_{K,\phi}} \,\dd \nu_{K_m}     \Big|+Am^{-1}\\
	&=\Big| \int_X  U^\star_{\nu_{K,\phi}}  \,\dd  (\nu_{K,\phi}-\nu_{K_m})  \Big| +Am^{-1}=\Big| \int_X  U^\star_{\nu_{K,\phi}}  \,\ddc  (U^\star_{\nu_{K,\phi}}-U^\star_{\nu_{K_m}})  \Big| +Am^{-1}\\
	&= \Big| \int_X     (U^\star_{\nu_{K,\phi}}-U^\star_{\nu_{K_m}}) \,\ddc U^\star_{\nu_{K,\phi}}    \Big| +Am^{-1} =\Big| \int_X     (U^\star_{\nu_{K,\phi}}-U^\star_{\nu_{K_m}}) \,(\ddc U^\star_{\nu_{K,\phi}} +\omega)   \Big| +Am^{-1}\\
	&=\Big| \int_X     (U^\star_{\nu_{K,\phi}}-U^\star_{\nu_{K_m}}) \,\dd \nu_{K,\phi}     \Big| +Am^{-1}\leq 2Am^{-1}.
	\end{align*}
	
	For the integral of $\phi$,  Lemma \ref{lem-UK} says that $\phi=U_{K,\phi}$ on the support of $\nu_{K,\phi}$. So we have
	$$\int_X  \phi \,\dd \nu_{K,\phi}=\int_X  U_{K,\phi} \,\dd \nu_{K,\phi} \quad\text{and}\quad \int_X  \phi \,\dd \nu_{K_m}=\int_X  U_{K_m} \,\dd \nu_{K_m}.$$
    It follows that 
    $$ \Big|2 \int_X  \phi \,\dd \nu_{K,\phi}- 2 \int_X  \phi \,\dd \nu_{K_m}\Big|=2\Big| \int_X U_{K,\phi} \,\dd \nu_{K,\phi}-  \int_X U_{K_m} \,\dd \nu_{K_m}\Big|\leq 4Am^{-1}.    $$
	The definition of $\oI_\phi$ gives the first inequality of the lemma. The second inequality follows by Lemma \ref{lem-UK}.
\end{proof}

Now we can  finish the proof of Proposition \ref{prop-upper-functional}.

\begin{proof}[Proof of Proposition \ref{prop-upper-functional}]
	By Lemma \ref{prop-find-p-1}, for any $\mathbf p\in K^m$, we have 
	$$ -\oJ_{m,\phi}(\mathbf p) \leq -\oI_\phi(\sigma_{\mathbf p}) + O(1+\norm{\phi}_{\Cc^\gamma}) \frac{\log m}{m}.$$
	Since $\supp(\sigma_{\mathbf p})\subset K_m$, giving $\min_{\cM(K_m)}\oI_\phi \leq \oI_\phi(\sigma_{\mathbf p})$. Combining with Lemma \ref{lem-Km} We deduce that 
	$$-\oJ_{m,\phi}(\mathbf p) \leq -\min_{\cM(K_m)}\oI_\phi+O(1+\norm{\phi}_{\Cc^\gamma}) \frac{\log m}{m}\leq   -\min_{\cM(K)}\oI_\phi+2O(1+\norm{\phi}_{\Cc^\gamma}) \frac{\log m}{m}. $$
	The complete the proof of the proposition.
\end{proof}

\section{Optimal lower bound of $\oJ_{m,\phi}$}\label{sec-com-2}

In Section \ref{sec-com-1}, we proved that $\min_{\cM(K)}\oI_\phi$ is a lower bound of $\oJ_{m,\phi}$ over $K^m$. 
In this section, we will show that it is optimal. In other words, we will construct a configuration $\mathbf{p} \in K^m$ that satisfies
\begin{equation*}\label{exist-p}
\oJ_{m,\phi}(\mathbf p) \leq \min_{\cM(K)}\oI_\phi + O\Big(\frac{\log m}{m}\Big).
\end{equation*}
The main idea comes from Nishry-Wennman \cite[Appendix]{Nishry-Wennman}. But the settings are quite different, where they work on $\C$. Here, we need to overcome the difficulties arising from the compactness of $X$ as well as the curvature of $\omega$.

\medskip

For   $m > 1$ and $\mathbf p \in X^m$, define the auxiliary  functional
\begin{equation}
\label{define K_m}
\oK_{m,\phi} (\mathbf p) := -\frac{m}{m-1} \oE_m (\mathbf p) + 2 \int_X \phi \, \dd \delta_{\mathbf p}.
\end{equation}
It only differ from $\oJ_{m,\phi}$   by a factor ${m\over m-1}$ on  $\oE_m (\mathbf p)$, whose appearance is quite important (see Lemma \ref{lem-fekete} below).
Let $\mathbf f_m =( f_{m, 1},\dots,f_{m,m})$ be a minimizer of $\oK_{m,\phi}(\mathbf p)$ for  $\mathbf p\in K^m$. It may not be unique, but we fix one. Clearly, $f_{m,j}$'s are  distinct. For each $j$, We denote 
$$\mathbf f'_{m,j}:=(f_{m, 1},\dots, f_{m, j-1},f_{m,j+1},\dots,f_{m,m})\quad\text{and}\quad \delta_{\mathbf f'_{m,j}}:={1\over m-1} \sum_{k\neq j}\delta_{f_{m,k}}  .$$

\begin{lemma}\label{lem-fekete}
For each $1\leq j\leq m$,  we have
	\begin{equation}\label{ineq-f-0}
	-U^\star_{\nu_{K,\phi}}(f_{m,j}) + \phi(f_{m,j}) = \min_{\cM(K)}\oI_\phi - \int_X \phi \, \dd \nu_{K,\phi} ;
	\end{equation}
	\begin{equation}	\label{ineq-f-1}
	-U^\star_{\delta_{\mathbf f'_{m,j}}}(f_{m, j}) + \phi(f_{m, j}) \leq -U^\star_{\delta_{\mathbf f'_{m,j}}}(z) + \phi(z)  \quad\text{for all }\, z \in K.
\end{equation}
\end{lemma}

\begin{proof}
We only show the lemma for $j=1$. The other cases can be treated in the same way.
 Define the function
	\[ \Phi(z) := \oK_{m,\phi}(z, f_{m, 2} , \dots , f_{m, m}) \quad\text{for }\, z\in K, \]
 who attains a minimum  at $z = f_{m, 1}$. After removing constant terms from $\Phi$, we get a new function (recall $G(z, f_{m, j})=G(f_{m, j},z)$)
		\[ \widetilde{\Phi}(z) := -\frac{m}{m-1} \frac{2}{m^2} \sum_{2\leq j \leq m} G(z, f_{m, j}) + \frac{2}{m} \phi(z) = \frac{2}{m} \big( -U^\star_{\delta_{\mathbf f'_{m,1}}}(z) + \phi(z) \big) \]
	who also attains its minimum over $K$ at $z = f_{m, 1}$. In other words,
	$$
	-U^\star_{\delta_{\mathbf f'_{m,1}}}(f_{m, 1}) + \phi(f_{m, 1}) \leq -U^\star_{\delta_{\mathbf f'_{m,1}}}(z) + \phi(z)  \quad\text{for all }\, z \in K.
	$$
	This gives the second assertion. 
		For the first one,
		By the second inequality in Proposition \ref{sol-min},
	and using that $U_{K,\phi}\leq \phi$ on $K$, we get 
	$$  U^\star_{\nu_{K,\phi}}-\int_X \phi \,\dd\nu_{K,\phi}+\min_{\cM(K)}\oI_\phi\leq \phi \quad\text{on }\, K.$$
	So, it remains to show that 
	\begin{equation}
	\label{leq-lemma51}
	-U^\star_{\nu_{K,\phi}}(f_{m, 1}) + \phi(f_{m, 1}) \leq \min \oI_\phi - \int_X \phi \, \dd \nu_{K,\phi}
	\end{equation}
Consider another auxiliary function \[\Psi(z) := U^\star_{\delta_{\mathbf f'_{m,1}}}(z) - U^\star_{\delta_{\mathbf f'_{m,1}}}(f_{m, 1}) + \phi(f_{m, 1}).\]  It is $\omega$-subharmonic and bounded from above by $\phi$ on $K$. By the definition of $U_{K,\phi}$, It should also be bounded by $U_{K,\phi}$ on $X$. Therefore, using Proposition \ref{sol-min} again, we have
	\[ \Psi(z)\leq U_{K,\phi}(z) =U^\star_{\nu_{K,\phi}}(z)-\int_X \phi \,\dd\nu_{K,\phi}+\min_{\cM(K)}\oI_\phi\]
	Taking $z = f_{m, 1}$, we obtain \eqref{leq-lemma51} and finish the proof of the lemma.
\end{proof}

In the proof above, one should see that in the definition of $\oK_{m,\phi}$, the factor ${m\over m-1}$ cannot be dropped.

\begin{lemma}\label{sep-fm}
For all $m>1$, we have
$$ 
	\dist(f_{m, j}, f_{m, k}) \geq \beta m^{-1/\gamma}  \, \text{ if }\,   j \neq k, 
$$
where $\beta$ depends on $\norm{\phi}_{\Cc^\gamma}$,  but not on $\phi$ and $m$.
\end{lemma}

\begin{proof}
	Without loss of generality, for every $m$, we assume $\dist(f_{m,1},f_{m,2})$ is the minimum among all $\dist(f_{m,j},f_{m,k}),j\neq k$. Suppose for contradiction, $m^{1/\gamma}\dist(f_{m,1},f_{m,2})$ tends to $0$ as $m\to \infty$. After taking a subsequence, we assume $f_{m,1},f_{m,2}$ both converge to some point $y\in K$. Taking a small open ball $B\ni y$.
	
	By Proposition \ref{sol-min} and Lemma \ref{lem-fekete}, we have for all $z\in K$,
	$$
	 U^\star_{\delta_{\mathbf f'_{m,j}}}(z) -U^\star_{\nu_{K,\phi}}(z) \leq U^\star_{\delta_{\mathbf f'_{m,j}}}(f_{m, j}) - U^\star_{\nu_{K,\phi}}(f_{m, j}). $$
	The function $U^\star_{\delta_{\mathbf f'_{m,j}}}(z) -U^\star_{\nu_{K,\phi}}$ is harmonic on $X\setminus K$. By maximal modulus principle, last inequality holds for all $z\in X$. Recall that $U^\star_{\nu_{K,\phi}}$ is $\gamma$-H\"older from Lemma \ref{lem-UK}. We obtain for all $z\in X$, 
	\begin{equation}\label{UzlessUf}
	U^\star_{\delta_{\mathbf f'_{m,1}}}(z) -U^\star_{\delta_{\mathbf f'_{m,1}}}(f_{m, 1})\leq \norm{U^\star_{\nu_{K,\phi}}}_{\Cc^\gamma}\dist(z,f_{m, 1})^\gamma \leq c\norm{\phi}_{\Cc^\gamma}\dist(z,f_{m, 1})^\gamma. \end{equation}
	
	Now using \eqref{sameweight}, for each $m$ and $j$, we take $\omega$-admissible metric $\fh_{m,j}$ for the line bundle $\Oc(f_{m,j})$, such that $$|\mathbf 1_{\Oc(f_{m,j})}|_{\fh_{m,j}} = |g_{m,j}| e^{-\varphi}   \quad\text{on }\, B,$$ 
	where $g_{m,j}$ is holomorphic on $B$ representing the holomorphic section having simple zero at $f_{m,j}$ and $\ddc \varphi=\omega$ on $B$. The formula \eqref{bundlepotential} gives that on $B$, one has
	$$  U^\star_{\delta_{\mathbf f'_{m,1}}}= {1\over m-1}\sum_{2\leq j\leq m} \big(\log|\mathbf 1_{\Oc(f_{m,j})}|_{\fh_{m,j}}+c_{\fh_{m,j}}\big)={1\over m-1}\sum_{2\leq j\leq m} \big( \log|g_{m,j} e^{-\varphi}|+c_{\fh_{m,j}}\big).  $$
	Substituting it into \eqref{UzlessUf}, we obtain for $z\in B$,
	$${1\over m-1}\sum_{2\leq j\leq m}  \Big(   \log \big|g_{m,j}(z) e^{-\varphi(z)}\big| -\log \big|g_{m,j}(f_{m, 1}) e^{-\varphi(f_{m, 1})}\big|  \Big)\leq  \norm{U^\star_{\nu_{K,\phi}}}_{\Cc^\gamma}\dist(z,f_{m, 1})^\gamma.$$  
	Now we restrict $z$ into the smaller open ball $\B(f_{m,1}, (m-1)^{-1/\gamma} )\subset B$. Then last inequality becomes
	$$ \Big|\prod_{2\leq j \leq m} { g_{m,j}(z)\over g_{m,j}(f_{m, 1}) } \Big| \leq  c\norm{\phi}_{\Cc^\gamma}\, e^{(m-1)(\varphi(z) -\varphi(f_{m, 1}))}   \quad\text{for }\,     z\in \B(f_{m,1}, (m-1)^{-1/\gamma} ).      $$
	The smoothness of $\varphi$ gives that $(m-1)(\varphi(z)-\varphi(f_{m,1}))$ is bounded on $ \B(f_{m,1}, (m-1)^{-1/\gamma} )$. Consider the holomorphic function 
	$$   \Phi(z):=  \prod_{2\leq j \leq m} { g_{m,j}(z)\over g_{m,j}(f_{m, 1}) }   \quad\text{on }\,  \B(f_{m,1}, (m-1)^{-1/\gamma} ).$$
	
	Remind that we assume $m^{1/\gamma}\dist(f_{m,1},f_{m,2})$ tends to $0$, which means $f_{m,2}\in  \B(f_{m,1}, (m-1)^{-1/\gamma} )$ for $m$ large enough. Thus, $\Phi$ is a holomorphic function on $\B(f_{m,1}, (m-1)^{-1/\gamma} )$, whose modulus is bounded by some constant independent of $m$, such that $\Phi(f_{m,1})=1, \Phi(f_{m,2})=0$ and $\dist(f_{m,1},f_{m,2})\to 0$. This contradicts to Schwarz lemma.
\end{proof}

The following is the main result of this section.

\begin{proposition}\label{prop-lower}
	There exists a constant $C>0$ independent of $\phi$ and $m$, but depending on $\norm{\phi}_{\Cc^\gamma}$, such that 
	$$   \oJ_{m,\phi}(\mathbf f_m) \leq \min_{\cM(K)}\oI_\phi + C\frac{\log m}{m} \quad\text{for all}\quad m>1.$$
\end{proposition}

Recall that in Section  \ref{sec-com-1}, we have defined a probability measure $\sigma_{\mathbf p}$ on $X$. By a similar  way, for each $1\leq j\leq m$, we define a probability measure $\sigma_{m,j}$ as follows: consider the function
\[
V_{m,j} := \sup_\varphi  \big\{  \varphi\in \SH(X,\omega): \; \varphi \leq U^\star_{\delta_{f_{m,j}}} \,\text{ on }\, \partial  \B(f_{m,j}, m^{-2/\gamma}) \big\}.
\]
Define the probability measure $$\sigma_{m,j} := \ddc V_{m,j} + \omega,$$ which is supported on $\partial \B(f_{m,j}, m^{-2/\gamma})$. Finally, we define their average 
$$    \sigma_m:={1\over m}\sum_{1\leq j\leq m} \sigma _{m,j} .    $$
From Lemma \ref{sep-fm}, we see that when $m>(2/\beta)^\gamma$, the support of $\sigma_{m,j}$'s are disjoint with each other. Repeat the proof of Lemma \ref{lem-delta-sigma}, we get 

\begin{lemma} \label{lem-sigma-m-delta}
	For each $1\leq j \leq m$,  we have $|V_{m,j}-U^\star_{\sigma_{m,j}}|\leq Cm^{-4}\log m$ on $X$ and $|U^\star_{\sigma_{m,j}}-U^\star_{\delta_{f_{m,j}}} |\leq Cm^{-4}\log m$ on $X\setminus  \B(f_{m,j}, m^{-2/\gamma})$.
\end{lemma}

Then we get the following completion of Lemma \ref{prop-find-p-1}, with the lower bound.

\begin{lemma}\label{lem-oE-sigma}
		There exists a constant $C>0$ independent of $\phi$ and $m$ such that,
	\begin{equation*}\label{ineq-oebound}
	\Big|\oE_m(\mathbf f_m)-\int_X  U^\star_{\sigma_{m}}\,\dd \sigma_{m} \Big|\leq C {\log m \over m}  \quad\text{for all}\quad m>(2/\beta)^\gamma.
	\end{equation*}
\end{lemma}

\begin{proof}
The assumption $m>(2/\beta)^\gamma$  ensures that $\overline\B(f_{m,j},m^{-2/\gamma})$'s are disjoint with each other.
	For $j\neq k$,  Lemma \ref{lem-sigma-m-delta} says that
	$$ |U^\star_{\sigma_{m,j}}-U^\star_{\delta_{f_{m,j}}} |\leq Cm^{-4}\log m  \quad\text{on}\quad  \overline\B(f_{m,k},m^{-2/\gamma}). $$
	 Using  \eqref{commute-potential}, we have
	\begin{align*} 
	U^\star_{\delta_{f_{m,j}}}(f_{m,k})&= U^\star_{\sigma_{m,j}}(f_{m,k}) + O(m^{-4}\log m)= \int_X U^\star_{\sigma_{m,j}} \,\dd  \delta_{f_{m,k}}+ O(m^{-4}\log m) \\
	&=\int_X U^\star_{ \delta_{f_{m,k}}} \,\dd  \sigma_{m,j}+ O(m^{-4}\log m)=\int_X U^\star_{\sigma_{m,k}} \,\dd  \sigma_{m,j}+ 2O(m^{-4}\log m). 
	\end{align*}
	It follows from \eqref{exp-oE} that 
	$$ \oE_m(\mathbf f_m)={1\over m^2}\sum_{j\neq k}\int_X U^\star_{\sigma_{m,k}} \,\dd  \sigma_{m,j}+ O(m^{-4}\log m).    $$

	By definition of $\sigma_m$,
	$$\int_X U^\star_{\sigma_{m}}\,\dd \sigma_{m}  ={1\over m^2}\sum_{j,k}\int_X U^\star_{\sigma_{m,j}} \,\dd \sigma_{m,k}.$$
	Thus, we conclude that 
	$$ \Big|\oE_m(\mathbf f_m)-\int_X  U^\star_{\sigma_{m}}\,\dd \sigma_{m} \Big|\leq \Big| {1\over m^2} \sum_{1\leq j\leq m}\int_X U^\star_{\sigma_{m,j}} \,\dd \sigma_{m,j}\Big| + O(m^{-4}\log m).$$
The lemma just follows from \eqref{pkpk}.
\end{proof}

\begin{lemma}\label{lem-f-sigma}
	There exists a constant $c>0$ independent of $\phi$ and $m$ such that for all $m>1$,
	$$\Big| \int_X  U^\star_{\nu_{K,\phi}} \,\dd \delta_{\mathbf{f}_m}  - \int_X  U^\star_{\nu_{K,\phi}} \,\dd \sigma_m    \Big|\leq c \norm{\phi}_{\Cc^\gamma}m^{-2}.$$
\end{lemma}

\begin{proof}
	By definition, the left hand side is
	\begin{align*}
	{1\over m}\Big| \sum_{1\leq j\leq m}\int_X  U^\star_{\nu_{K,\phi}} \,\dd \delta_{f_{m,j}}  -\sum_{1\leq j\leq m} \int_X  U^\star_{\nu_{K,\phi}} \,\dd \sigma_{m,j}    \Big|.
	\end{align*}
	For each $j$, the distance from every point in $\supp(\sigma_{m,j})$ to $f_{m,j}$ is $m^{-2/\gamma}$. The lemma will follow from the fact that $\norm{U^\star_{\nu_{K,\phi}}}_{\Cc^\gamma}\leq c\norm{\phi}_{\Cc^\gamma}$ in Lemma \ref{lem-UK}.
\end{proof}

\begin{proof}[Proof of Proposition \ref{prop-lower}]
	Summing up all the $m$ inequalities in \eqref{ineq-f-1}, yields
	$$-\sum_{1\leq j\leq m}U^\star_{\delta_{\mathbf f'_{m,j}}}(f_{m, j}) +\sum_{1\leq j\leq m} \phi(f_{m, j}) \leq -\sum_{1\leq j\leq m}U^\star_{\delta_{\mathbf f'_{m,j}}}(z) + m\phi(z)  \quad\text{for all }\, z \in K.$$
	This is equivalent to (c.f.\ \eqref{exp-oE})
	$$  -{m^2\over m-1}\oE_m(\mathbf f_m)+m   \int_X \phi \,\dd \delta_{\mathbf f_m}  \leq - m  U^\star_{\delta_{\mathbf f_m}}(z) +m\phi(z) \quad\text{for all }\, z \in K.$$
	Integrating it against $\nu_{K,\phi}$ (since $\supp(\nu_{K,\phi})\subset K$) gives
	$$ -{m^2\over m-1}\oE_m(\mathbf f_m)+ m  \int_X \phi \,\dd \delta_{\mathbf f_m}  \leq - m \int_X U^\star_{\delta_{\mathbf f_m}} \,\dd\nu_{K,\phi} +m\int_X\phi \,\dd\nu_{K,\phi} . $$
	Multiply by ${m-1\over m^2}$ on both sides and using \eqref{commute-potential}, we get
$$  -\oE_m (\mathbf f_m)+ {m-1\over m}\int_X \phi \,\dd \delta_{\mathbf f_m} \leq  -{m-1\over m} \int_X U^\star_{\nu_{K,\phi}}\,\dd \delta_{\mathbf f_m} +{m-1\over m}\int_X \phi \,\dd \nu_{K,\phi}.  $$
Since $\norm{U^\star_{\nu_{K,\phi}}}_{\Cc^\gamma}\leq c\norm{\phi}_{\Cc^\gamma}$ by Lemma \ref{lem-UK}, the three integrals in the inequality are all bounded independent of $m$. Therefore,
$$ -\oE_m (\mathbf f_m)+  \int_X \phi \,\dd \delta_{\mathbf f_m}\leq -\int_X U^\star_{\nu_{K,\phi}}\,\dd \delta_{\mathbf f_m} + \int_X \phi \,\dd \nu_{K,\phi}+O(\norm{\phi}_{\Cc^\gamma})m^{-1}.  $$
Applying Lemma \ref{lem-f-sigma} and multiplying by a factor $2$, we get
\begin{equation} \label{eq5.2}
-2\oE_m (\mathbf f_m)+ 2 \int_X \phi \,\dd \delta_{\mathbf f_m}\leq -2\int_X U^\star_{\nu_{K,\phi}}\,\dd \sigma_m +2 \int_X \phi \,\dd \nu_{K,\phi} +O(\norm{\phi}_{\Cc^\gamma})m^{-1}.   
\end{equation}

On the other hand, applying \eqref{positive energy} to the signed measure $\nu_{K,\phi}-\sigma_m$, yields
$$ -2\int_X U^\star_{\nu_{K,\phi}}\,\dd  \sigma_m \leq  -\int_X U^\star_{\nu_{K,\phi}}\,\dd \nu_{K,\phi} - \int_X U^\star_{\sigma_m}\,\dd \sigma_m. $$
Hence, together with \eqref{eq5.2}, we conclude that  
$$-2\oE_m (\mathbf f_m)+ 2 \int_X \phi \,\dd \delta_{\mathbf f_m} \leq   -\int_X U^\star_{\nu_{K,\phi}}\,\dd \nu_{K,\phi} - \int_X U^\star_{\sigma_m}\,\dd \sigma_m +2 \int_X \phi \,\dd \nu_{K,\phi}+ O(\norm{\phi}_{\Cc^\gamma})m^{-1}. $$
By invoking Lemma \ref{lem-oE-sigma}, and the fact that $\beta$ depends on $\norm{\phi}_{\Cc^\gamma}$, we finish the proof of Proposition \ref{prop-lower}.
\end{proof}

\section{Riemann-Theta section}\label{sec-rie}

In the decomposition of $|\det S|_h$ in Lemma \ref{boson}, besides the exponential of the Green functions, the Riemann-Theta section also appears. Put
$$ \oR_m(x_1,\dots,x_m):= \Big\|\widetilde\theta\big(  \oA_{n}(L^{m+g-1})-\sum_{1\leq j\leq m}\oA_1(x_j) -\mathbf z_\star \big)\Big\|_{\fh_J}^2. $$
Our aim is to  find the upper bound of $|\det S|_h$.
In Section \ref{sec-com-2}, for every $m$, we have found a special point $\mathbf f_m\in K^m$, such that the value of $\oJ_{m,\phi}(\mathbf f_m)$ is small enough, which is close to the minimal value of the functional $\oI_\phi$ over $\cM(K)$.   However, it may happen that the value of $\oR_m(\mathbf f_m)$ is quite close to zero, in which case, the value of $|\det S|_h$ will also be quite small. In this section, we will do a perturbation with the point $\mathbf f_m$, getting a new point $\mathbf F_m\in K^m$, such that the value $\oR_m(\mathbf F_m)$ is bounded uniformly from below. Moreover, the new functional value $\oJ_{m,\phi}(\mathbf F_m)$ does not oscillate too much.

\begin{lemma}\label{lem-pet-fekete}
	 There exist constants $\alpha,\kappa>0$ independent of $m$ such that for any $\mathbf Z \in \mathrm{Jac}(X)$, any $m>1$ and any $m$ points $x_1,x_2,\dots,x_m\in X$, one can find $\mathbf y:=(y_1,y_2,\dots,y_g)\in  K^g$, satisfying the following inequalities:
	$$  \dist(y_j,x_\ell) \geq \alpha/\sqrt m \quad\text{for all}\quad j,\ell,$$
		$$   \dist(y_j,y_k)\geq \alpha  \quad\text{for}\quad  j\neq k,        $$ 
	$$ \big \|\widetilde\theta\big(  -\oA_g(\mathbf y)+\mathbf Z \big)\big\|_{\fh_J}^2   \geq \kappa.  $$
\end{lemma}

\begin{proof}
	
	Recall that $\Theta\subset \mathrm{Jac}(X)$ is the zero locus of $\widetilde\theta$ and $\W$ is the critical set of $\oA_g$. Clearly, $\W\cup \oA_g^{-1}(\mathbf Z-\Theta)$ is an analytic subset of $X^g$.   Since $K$ is the closure of an open set, we can fix an open ball $\mathbf B\subset K^g$ such that 
	$$ \diam(\mathbf B)= \alpha_1,\quad\text{and}\quad \dist\big( \mathbf B  \, , \, \W\cup\oA_g^{-1}( \mathbf Z-\Theta)           \big)\geq \alpha_1 \quad \text{ in }\,X^g  $$
	for some small $\alpha_1>0$ independent of $m,\mathbf Z$. Then we take   $B_j,1\leq j\leq g$ in the $j$-th projection of $\mathbf B$ such that $\diam(B_j)=\alpha_2$ and $\dist(B_j,B_k)\geq \alpha_2$ for $j\neq k$. This    $\alpha_2>0$   only depends on $\alpha_1$.

	For each $1\leq j\leq g$, by an area computation argument, there exists some small $\alpha_3>0$, such that the open balls $\B(x_\ell,\alpha_3/\sqrt m),1\leq \ell \leq m$ cannot cover $B_j$. So, we can take a point
	$$y_j\in B_j\setminus \cup_{1\leq \ell \leq m}\B(x_\ell ,\alpha_3/\sqrt m).$$
	
	Now we check that the $y_k$'s satisfy the three requirements. The first one follows from the definition of $y_j$ with $\alpha=\alpha_3$. The second one also follows from the definition of $y_j$ with $\alpha=\alpha_2$. For the last one,  since $\dist(\mathbf y,\W)\geq \dist(\mathbf B,\W)\geq \alpha_1$, the open ball $\B(\mathbf y, \alpha_1/2)$ is biholomorphic to $\oA_g(\B(\mathbf y, \alpha_1/2))$, whose diameter in $\mathrm{Jac}(X)$ is larger than some $\alpha_4>0$. Using that $\dist \big(\mathbf y,\oA_g^{-1}( \mathbf Z-\Theta) \big)\geq \dist \big(\mathbf B,\oA_g^{-1}( \mathbf Z-\Theta) \big)\geq\alpha_1$, we see that in $\mathrm{Jac}(X)$,
	 $$\dist \big(- \oA_g(\mathbf y)+\mathbf Z,\Theta\big)=\dist \big( -\oA_g(\mathbf y),\Theta-\mathbf Z\big)=\dist \big( \oA_g(\mathbf y),\mathbf Z-\Theta\big)\geq \alpha_5$$ for some $\alpha_5>0$. This says the value  $\big \|\widetilde\theta\big( - \oA_g(\mathbf y)+\mathbf Z \big)\big\|_{\fh_J}^2 $ is bounded from below by some constant $\kappa>0$ because $\Theta$ is the zero locus of $\widetilde\theta$. So we finish the proof of the lemma by taking $\alpha:=\min(\alpha_2,\alpha_3)$.
\end{proof}

\begin{lemma} \label{lem-new-fekete}
	For every $m>g$, there exist $\alpha,\kappa>0$ independent of $m$ and a point $\mathbf F_m :=(F_{m,1},\dots,F_{m,m})\in K^m$ such that $$F_{m,\ell} =f_{m,\ell} \quad\text{for}\quad g+1\leq \ell\leq m ,$$
	$$ \dist(F_{m,j},F_{m,\ell}) \geq  \alpha/\sqrt m \quad\text{for}\quad   1\leq j\leq g,\, g+1\leq \ell\leq m,    $$
	$$ \dist(F_{m,j},F_{m,k})\geq \alpha\quad\text{for}\quad  1\leq j\neq k\leq g,$$
$$\oR_m(\mathbf F_m)\geq \kappa.$$
\end{lemma}

\begin{proof}
	Applying Lemma \ref{lem-pet-fekete} with 
	$$(x_1,\dots,x_m)=\mathbf f_m,\quad \mathbf Z=  \oA_{n}(L^{m+g-1})-\sum_{g+1\leq j\leq m}\oA_1(f_{m,j})  -\mathbf z_\star,$$
	we will get a point $\mathbf y\in K^g$. Let $\mathbf F_m:=(y_1,\dots,y_g,f_{m,g+1}\dots,f_{m,m})$, which gives all the desired inequalities.
\end{proof}

We conclude  from Lemmas \ref{sep-fm} and  \ref{lem-new-fekete} that for $j\neq k$, 
  $$ \dist(f_{m,j} ,f_{m,k})\geq   \beta m^{-1/\gamma}  \quad\text{and}\quad  \dist(F_{m,j} ,F_{m,k})\geq \min(\alpha,\beta)  m^{-1/\gamma}.     $$
After decreasing the value of $\beta$, we will still assume
\begin{equation}\label{sep-FM}
\dist(F_{m,j} ,F_{m,k})\geq \beta  m^{-1/\gamma} \quad\text{for }\, j\neq k.
\end{equation}

\begin{proposition}\label{prop-Fmfm}
	There exists a constant $C>0$ independent of $\phi$ and $m$, but depending on $\norm{\phi}_{\Cc^\gamma}$, such that for all $m>g$,
	$$  \big|\oJ_{m,\phi}(\mathbf F_m)-\oJ_{m,\phi}(\mathbf f_m) \big|\leq C{\log m\over m}.      $$
\end{proposition}

\begin{proof}
	Notice that $\mathbf F_m$ and $\mathbf f_m$ only differ by $g$ coordinates. So,
	$$  \Big|\int_X \phi \,\dd \delta_{\mathbf F_m} -\int_X \phi \,\dd\delta_{\mathbf f_m}\Big| =\Big|{1\over m}  \sum_{1\leq j\leq g} \big(\phi(F_{m,j})-\phi(f_{m,j}) \big) \Big| \leq {2g\over m}\norm{\phi}_\infty. $$ 
	It remains to estimate $\oE_m(\mathbf F_m)-\oE_m(\mathbf f_m)$. By definition it is equal to
	$$ {1\over m^2}\sum_{1\leq j\neq k\leq m} \Big(G(F_{m,j},F_{m,k}) -G(f_{m,j},f_{m,k}) \Big).  $$
    There are $m(m-1)$ terms in the summation. When $j$ and $k$ are both larger than $g$, 
   $G(F_{m,j},F_{m,k}) -G(f_{m,j},f_{m,k})=0$. So, only $m(m-1)-(m-g)(m-g-1)$ terms are nonzero. Moreover, we know that for $j\neq k$, 
   $ \dist(f_{m,j} ,f_{m,k}),\dist(F_{m,j} ,F_{m,k})$ are both larger than $\beta  m^{-1/\gamma}$.
    It follows that $|G(F_{m,j},F_{m,k})|$ and $|G(f_{m,j},f_{m,k})|$ are both $O(\log m)$. Thus, we conclude that 
   $$ \big| \oE_m(\mathbf F_m)-\oE_m(\mathbf f_m) \big|=O(m^{-1}\log m). $$
   This finishes the proof of the proposition.
\end{proof}

The following result will be used later.

\begin{lemma}\label{lem-omega}
	If $\mathbf p=(p_1,\dots,p_m)\in X^m$ satisfying $\dist(p_j,F_{m,j})\leq  m^{-2/\gamma}$ for all $1\leq q\leq m$, then there is a constant $C>0$ independent of $\phi$ and $m$, but depending on $\norm{\phi}_{\Cc^\gamma}$, such that
		$$  \big|\oJ_{m,\phi}(\mathbf p) -\oJ_{m,\phi}(\mathbf F_m) \big|\leq C m^{-1}   \quad\text{for all }\, m>1.      $$
		and $\oR_m(\mathbf p)\geq \kappa/2$, where $\kappa$ is the constant in Lemma \ref{lem-new-fekete}.
\end{lemma}

\begin{proof}
Since $\phi$ is $\gamma$-H\"older, 
$$  \Big|\int_X \phi \,\dd \delta_{\mathbf p} -\int_X \phi \,\dd\delta_{\mathbf F_m}\Big| \leq {1\over m}\sum_{1\leq j\leq m} \Big| \int_X \phi \,\dd \delta_{p_j} -\int_X \phi \,\dd\delta_{F_{m,j}}\Big|\leq \norm{\phi}_{\Cc^\gamma} (m^{-2/\gamma})^\gamma. $$
For the first assertion, it remains to show $ \big| \oE_m(\mathbf p)-\oE_m(\mathbf F_m) \big|=O(m^{-1}\log m)$. By definition of $\oE_m$, we only need to
 prove that for every $j\neq k$,
$$  \big| G(p_j,p_k)- G(F_{m,j},F_{m,k}) \big|  =O(m^{-1}).  $$
We use Lemma \ref{l:Green} to bounded the left hand side by  
$$  \big| \log\dist(p_j,p_k)-\log\dist(F_{m,j},F_{m,k})\big|+\big|   \varrho(p_j,p_k)-\varrho (F_{m,j},F_{m,k})      \big|.$$
Since $\varrho$ is Lipschitz, the second term is $O(m^{-2/\gamma})$. Using that $\dist(p_j,p_k)\leq \dist(F_{m,j},F_{m,k})+2m^{-2/\gamma}$, we see that the first term is bounded by
$$  \log {  \dist(F_{m,j},F_{m,k})+2m^{-2/\gamma}\over   \dist(F_{m,j},F_{m,k})} \leq \log { \beta m^{-1/\gamma}+2m^{-2/\gamma}\over   \beta m^{-1/\gamma}} =O(m^{-1/\gamma}).$$
This gives the first assertion. For the second one, by definition of $\oA_m$ and using that the distance on $\mathrm{Jac}(X)$ is flat locally, we have 
\begin{align*}
&\dist\Big(\oA_m(\mathbf p),\oA_m(\mathbf F_m) \Big)=\dist \Big(\sum_{1\leq j\leq m}\oA_1(p_j) ,\sum_{1\leq j\leq m}\oA_1(F_{m,j}) \Big) \\
&\leq \sum_{1\leq j\leq m}\dist \big( \oA_1(p_j), \oA_1(F_{m,j})  \big) \leq  \sum_{1\leq j\leq m}\norm{\oA_1}_{\Cc^1}\dist(p_j,F_{m,j})\\
&\leq m  \norm{\oA_1}_{\Cc^1}m^{-2/\gamma}=O(m^{-2/\gamma +1}).
\end{align*} 
This tends to $0$ as  $m\to\infty$. The continuity of $\widetilde \theta$ gives the second assertion.
\end{proof}

\section{Estimates on $\mathcal Z_n$}

The constant $\mathcal Z_n$ appears in the decomposition of $\big|\det S_n\big|_h^2$ in  Lemma \ref{boson}.
In this section, we will prove

\begin{proposition}\label{prop-Zn}
	There exists a constant $C>0$ independent of $\phi$ and $m$, but depending on $\norm{\phi}_{\Cc^\gamma}$, such that for all $m>1$,
	$$-C N_n\log N_n\leq \log \mathcal Z_n \leq C N_n \log N_n.$$
\end{proposition}

The proof will rely on the lower bound of $\oJ_{m,\phi}$ we obtained in last three sections.
Before proving the proposition, we simplify $\norm{\det S_n}_{L^2(\mu,n\phi)}^2$ first.

\begin{lemma}\label{lem-simp}
	As $n\to \infty$, we have
	$$\norm{\det S_n}_{L^2(\mu,n\phi)}^2=  \mathcal Z_n  \int_{K^{N_n}}   \exp \Big[-N_n^2 \oJ_{N_n,\phi}(\mathbf x)+  O(N_n)   \Big] \oR_{N_n}(\mathbf x)\,\dd  \mu^{N_n}(\mathbf x),$$
	where the implicit factor of $O(N_n)$ depends on $\norm{\phi}_{\Cc^\gamma}$.
\end{lemma}

\begin{proof}
By definition and Lemma  \ref{boson}, 
\begin{align*}
\norm{\det S_n}_{L^2(\mu,n\phi)}^2 =\int_{K^{N_n}} &\mathcal Z_n e^{\sum_{1\leq \ell\leq N_n}\lambda(x_\ell)}  e^{\sum_{j\neq k} G(x_j,x_k)-2n\sum_{1\leq j\leq N_n} \phi(x_j)}  \\
& \Big\|\widetilde\theta\big(  \oA_{n}(L^{n})-\sum_{1\leq j\leq N_n}\oA_1(x_j) -\mathbf z_\star \big)\Big\|_{\fh_J}^2 \,  \dd \mu^{N_n}(x_1,\dots,x_{N_n}).   
\end{align*}
We rewrite it using \eqref{e-e} and our new defining functionals as
$$  \int_{K^{N_n}} \mathcal Z_n e^{\sum_{1\leq \ell\leq N_n}\lambda(x_\ell)}  \exp \Big[ N_n^2  \oE_{N_n}(\mathbf x) - 2nN_n\int_X \phi \,\dd \delta_{\mathbf x}      \Big] \oR_{N_n}(\mathbf x)\,\dd  \mu^{N_n}(\mathbf x) .   $$
Since $\phi$ is bounded independent of $n$ and $N_n=n-g+1$,
$$N_n^2  \oE_{N_n}(\mathbf x) - 2nN_n\int_X \phi \,\dd \delta_{\mathbf x} = -N_n^2 \oJ_{N_n,\phi}(\mathbf x)+  O(N_n)\norm{\phi}_{\Cc^\gamma}.    $$
Moreover, 
$$ \Big|\sum_{1\leq \ell\leq N_n}\lambda(x_\ell)\Big|\leq N_n\big|\max\lambda\big|=O(N_n)$$
because $\lambda$ is bounded independent of $n$. It follows that 
$$
\norm{\det S_n}_{L^2(\mu,n\phi)}^2=  \mathcal Z_n  \int_{K^{N_n}}   \exp \Big[-N_n^2 \oJ_{N_n,\phi}(\mathbf x)+  O(N_n)   \Big] \oR_{N_n}(\mathbf x)\,\dd  \mu^{N_n}(\mathbf x),
$$
proving the lemma.
\end{proof}

\begin{proof}[Proof of Proposition \ref{prop-Zn}-lower bound]
	Recall that $S_n$ is an $L^2(\omega ,n\mathbf{0})$-orthonormal basis of $H^0(X,L^n)$.
	By Lemma \ref{l2formula},
	$$ \norm{\det S_n}_{L^2(\omega,n\mathbf{0})}^2 =N_n ! \det \big( \lp  s_j,s_k \rp_{L^2(\omega,n\mathbf{0})}  \big)_{j,k}=N_n!.  $$
	Applying Lemma \ref{lem-simp} with $K=X,\phi=\mathbf 0,\mu=\omega$, we get
	\begin{equation}\label{Zn0}
	\int_{X^{N_n}}   \mathcal Z_n \, \exp \Big[-N_n^2 \oJ_{N_n,\phi}(\mathbf x)+  O(N_n)   \Big] \,\oR_{N_n}(\mathbf x) \,\omega^{N_n}(\mathbf x)=N_n!.
	\end{equation} 
	
	On the other hand,
applying Proposition \ref{prop-upper-functional} with $m=N_n,K=X$ and $\phi=\mathbf{0}$, we obtain for all $x\in X^{N_n}$,
$$ - \oJ_{N_n,\mathbf{0}}(\mathbf x)  \leq  -\min_{\cM(X)}\oI_{\mathbf{0}} +O(N_n^{-1}\log N_n).$$
Recall that $\min_{\cM(X)}\oI_{\mathbf{0}}=0$ from \eqref{minI0}. Using that  $\oR_{N_n}$ is continuous, we deduce from \eqref{Zn0} that 
$$ N_n!\leq \int_{X^{N_n}}  \mathcal Z_n  e^{N_n^2O(N_n^{-1}\log N_n)}  \max \oR_{N_n} \,\omega^{N_n}(\mathbf x)= \mathcal Z_n \, e^{O(N_n\log N_n)} \, \max \oR_{N_n} .  $$
This gives 
$$\log \mathcal Z_n \geq \log N_n! -O(N_n\log N_n)-\log \max \oR_{N_n} .$$
The lower bound follows.
\end{proof}

\begin{proof}[Proof of Proposition \ref{prop-Zn}-upper bound]
We start with the configuration  $\mathbf F_m$ obtained in Section \ref{sec-rie}. Let 
\begin{equation}\label{defn-omega}
\Omega_m:= \prod_{1\leq j\leq m} \B(F_{m,j}, m^{-2/\gamma}) \subset X^m. 
\end{equation}
 For every  $\mathbf p\in \Omega_m$, by Lemma \ref{lem-omega}, we have
 $$  \big|\oJ_{m,\phi}(\mathbf p) -\oJ_{m,\phi}(\mathbf F_m) \big|= O( m^{-1})  \quad\text{and}\quad \oR(\mathbf p)\geq \kappa/2.     $$
Together with Propositions \ref{prop-lower} and \ref{prop-Fmfm}, we obtain
\begin{equation}\label{omegam}
\oJ_{m,\phi}(\mathbf p) \leq  \min_{\cM(K)}\oI_\phi + O(m^{-1}\log m) \quad\text{and}\quad \oR(\mathbf p)\geq \kappa/2  \quad\text{for }\, \mathbf p\in \Omega_m.
\end{equation}
In particular, this holds for $m=N_n,K=X$ and $\phi=\mathbf{0}$, which says (recall $\min_{\cM(X)} \oI_{\mathbf 0}=0$)
$$\oJ_{N_n,\mathbf{0}}(\mathbf p) \leq O(N_n^{-1}\log N_n) \quad\text{and}\quad \oR(\mathbf p)\geq \kappa/2  \quad\text{for }\, \mathbf p\in \Omega_{N_n}. $$
  Shrinking the integral domain to $\Omega_{N_n}$ in \eqref{Zn0}, gives
$$\int_{\Omega_{N_n}}   \mathcal Z_n \, \exp \Big[-N_n^2 \oJ_{N_n,\phi}(\mathbf x)+  O(N_n)   \Big] \,\oR_{N_n}(\mathbf x) \,\omega^{N_n}(\mathbf x)\leq N_n!,  $$
which implies
$$\int_{\Omega_{N_n}}   \mathcal Z_n \,  e^{-N_n^2O(N_n^{-1}\log N_n)  } \,{\kappa\over 2} \,\omega^{N_n}\leq N_n!.    $$

Since $\omega$ is smooth and strictly positive, the $\omega$-area of an open ball of radius $r$ is at least $c_\omega r^2$ for some $c_\omega>0$ only depending on $\omega$. 
Using that $\Omega_{N_n}$ is the product of ${N_n}$ balls of radius ${N_n}^{-2/\gamma}$, we have
  $$\int_{\Omega_{N_n}}\,\omega^{N_n}\geq \big(c_\omega ({N_n}^{-2/\gamma})^2\big)^{N_n}= c_\omega^{N_n} N_n^{-4N_n/\gamma} .$$
Therefore, we conclude that 
$$  \log \mathcal Z_n \leq \log N_n! + O(N_n\log N_n)-\log(\kappa/2)-N_n\log c_\omega+4N_n/\gamma \log N_n.  $$
The upper bound follows.
\end{proof}

\section{Proof of the main theorems}

 We prove the theorems in the reverse order.

\begin{proof}[Proof of Theorem \ref{maintheorem3}]
	It is enough to assume $(K_2,\phi_2,\mu_2)=(X,\mathbf 0,\omega)$, which clearly satisfies the mass-density condition. 	Write $(K_1,\phi_1,\mu)$ as $(K,\phi,\mu)$.  We need to only show (recall $\min_{\cM(X)} \oI_{\mathbf 0}=0$)
	\begin{equation}\label{goal-theorem3}
	  \Big| \oL_n(\mu,\phi)-\oL_n(\omega ,\mathbf{0})     -\min_{\cM(K)} \oI_\phi \Big| = O\Big({\log n\over n}\Big)     
	\end{equation}
	if $\mu$ satisfies the mass-density condition on $K$.
	 The estimate in the theorem can be deduced by an easy triangular inequality. 
We need to investigate the difference $\oL_n(\mu,\phi)-\oL_n(\omega ,\mathbf{0})$. By \eqref{key-equation}, it is exactly
	 $$ {1\over n N_n}\log  \norm{\det S_n}_{L^2(\mu,n\phi)}^2 -{1\over n N_n} \log(N_n !).$$
	 
	 We now estimate the first term above.  
	 By Proposition \ref{prop-upper-functional}, for $\mathbf x\in K^{N_n}$,
	 $$ -N_n^2 \oJ_{N_n,\phi}(\mathbf x) \leq -N_n^2  \min_{\cM(K)} \oI_\phi +O(N_n\log N_n).$$
	 Therefore, taking the formula in Lemma \ref{lem-simp} into account, we have
	 $$ \norm{\det S_n}_{L^2(\mu,n\phi)}^2\leq  \mathcal Z_n    \exp \Big[-N_n^2  \min_{\cM(K)} \oI_\phi +O(N_n\log N_n)  \Big]\max\oR_{N_n}.   $$
	 Together with Proposition \ref{prop-Zn}, we obtain
	 \begin{align*}
	 {1\over n N_n}\log  \norm{\det S_n}_{L^2(\mu,n\phi)}^2  
	 &\leq  {1\over n N_n}\log \mathcal Z_n -{N_n^2\over nN_n}\min_{\cM(K)} \oI_\phi+{ O(N_n\log N_n)\over nN_n} + {\max\oR_{N_n}\over nN_n}\\
	 &\leq  { O(\log N_n)\over n}-{N_n\over n}\min_{\cM(K)} \oI_\phi.
	 \end{align*}
	 Recalling $N_n=n-g+1$, the first term is  $O(n^{-1}\log n)$. Our condition on $K$ ensure that $\min\oI_\phi$ is finite. So, 
	 \begin{equation}\label{Iphiterm}
	 -{N_n\over n}\min_{\cM(K)} \oI_\phi =-\min_{\cM(K)} \oI_\phi +{g-1\over n}\min_{\cM(K)} \oI_\phi = -\min_{\cM(K)} \oI_\phi+ O(n^{-1}).
	 \end{equation}
	 Thus, we can conclude that 
	 $$ {1\over n N_n}\log  \norm{\det S_n}_{L^2(\mu,n\phi)}^2 \leq  -\min_{\cM(K)} \oI_\phi +O\Big({\log n\over n} \Big).$$
	 This proves the upper bound in \eqref{goal-theorem3}. Notice in the proof so far, we do not need the mass-density condition.
	 
	 \smallskip
	 
	 For the lower bound in \eqref{goal-theorem3}, as in \eqref{defn-omega}, we define $$ \Omega_{N_n}:= \prod_{1\leq j\leq N_n} \B(F_{ N_n,j},  N_n^{-2/\gamma}) \subset X^{N_n}.    $$
	  It may not be contained in $K^{N_n}$. So we consider $\Omega_{N_n}\cap K^{N_n}$ and obtain from Lemma \ref{lem-simp} that
	  $$\norm{\det S_n}_{L^2(\mu,n\phi)}^2\geq  \mathcal Z_n  \int_{\Omega_{N_n}\cap K^{N_n}}  \exp \Big[-N_n^2 \oJ_{N_n,\phi}(\mathbf x) + O(N_n)   \Big] \oR_{N_n}(\mathbf x)\,\dd  \mu^{N_n}(\mathbf x) .$$
	 Using \eqref{omegam}, we see that 
	 $$\oJ_{N_n,\phi}(\mathbf x) \leq  \min_{\cM(K)}\oI_\phi + O(N_n^{-1}\log N_n)  \quad\text{and}\quad \oR(\mathbf x)\geq \kappa/2\quad\text{for }\, \mathbf x\in \Omega_{N_n}.$$
	 Hence
	  $$\norm{\det S_n}_{L^2(\mu,n\phi)}^2\geq  \mathcal Z_n \exp \Big[-N_n^2 \min_{\cM(K)}\oI_\phi - O(N_n\log N_n) \Big] {\kappa\over 2} \int_{\Omega_{N_n}\cap K^{N_n}}  \,\dd  \mu^{N_n}(\mathbf x) .  $$
	  We need to bound the last integral from below. Recall that the boundary of $K$ is $\Cc^2$. So, for $N_n$  large enough, the intersection $\B(F_{ N_n,j},  N_n^{-2/\gamma}) \cap K$ should contains a ball of radius $2^{-1}N^{-2/\gamma}$. The mass-density condition of $\mu$ on $K$ gives
	  $$ \mu \big( \B(F_{ N_n,j},  N_n^{-2/\gamma}) \cap K \big) \geq c (2^{-1}N_n^{-2/\gamma})^\tau   $$
	  for some $c>0$ and $\tau>0$.
	 It follows that
	  $$ \int_{\Omega_{N_n}\cap K^{N_n}}  \,\dd  \mu^{N_n}(\mathbf x) \geq  c^{N_n}  (2^{-1}N_n^{-2/\gamma})^{\tau N_n}. $$
	  Gluing all the inequality above and using Proposition \ref{prop-Zn}, we conclude that 
	  \begin{align*}  
	  &{1\over n N_n}\log  \norm{\det S_n}_{L^2(\mu,n\phi)}^2\\ 
	  &\geq  {1\over n N_n}\log\mathcal Z_n- {N_n^2 \over nN_n }\min_{\cM(K)}\oI_\phi - {O(N_n\log N_n) \over nN_n}-{ \log (\kappa/2) \over nN_n} +{N_n\log(c (2^{-1}N_n^{-2/\gamma})^\tau ) \over nN_n} \\
	  &\geq -{O(\log N_n)\over n }- {N_n  \over n  }\min_{\cM(K)}\oI_\phi\geq  -\min_{\cM(K)} \oI_\phi -O\Big({\log n\over n} \Big). 
	  \end{align*}
	  where we use \eqref{Iphiterm}.
	  This proves the lower bound in \eqref{goal-theorem3} and finishes the proof of Theorem \ref{maintheorem3}.
\end{proof}

\begin{proof}[Proof of Theorem \ref{maintheorem1}]
	Again, we assume  $(K_2,\phi_2,\mu_2)=(X,\mathbf 0,\omega)$ and write $(K_1,\phi_1,\mu_1)=(K,\phi,\mu)$. Lemma \ref{lem-K-X} allows us to replace $\oL_n(K,\phi)$ by $\oL_n(X,U_{K,\phi})$.
	We only need to prove
	\begin{equation}\label{goal-theorem1}
	  \Big|\oL_n(X,U_{K,\phi})-\oL_n(X,\mathbf 0) -  \min_{\cM(K)}\oI_{\phi}    \Big|= O\Big({\log n\over n}  \Big).  
	  \end{equation}
	
	Note that the statement of the theorem is independent of $\mu$. We can just let $\mu=\omega$.  Recall $U_{K,\phi}\in\Cc^\gamma(X)$ from Lemma \ref{lem-UK}.  So $(X,U_{K,\phi},\omega)$ satisfies the strongly Bernstein-Markov condition by Lemma \ref{lem-strongBM}.  It follows from Lemma \ref{lem-dmn} that
	 $$ 0\leq \oL_n(\omega,U_{K,\phi}) -\oL_n(X,U_{K,\phi})\leq O(n^{-1}\log n).  $$
	 By the same reason, we also have
	 $$ 0\leq \oL_n(\omega,\mathbf 0) -\oL_n(X,\mathbf 0)\leq O(n^{-1}\log n).  $$
	 
	 On the other hand, replacing $(K,\mu,\phi)$ by $(X,\omega,U_{K,\phi})$ in \eqref{goal-theorem3}, gives
	 $$\Big| \oL_n(\omega,U_{K,\phi})-\oL_n(\omega ,\mathbf{0})     -\min_{\cM(X)} \oI_\phi \Big| = O\Big({\log n\over n}\Big)$$
	 Gluing all last three inequalities together, we obtain \eqref{goal-theorem1} and finish the proof of Theorem \ref{maintheorem1}.
\end{proof}

\begin{proof}[Proof of Theorem \ref{maintheorem2}]
 Since $(K_1,\phi_1,\mu_1),(K_2,\phi_2,\mu_2)$  satisfy the strong Bernstein-Markov condition, Lemma \ref{lem-strongBM} gives
 	$$ 0\leq \oL_n(\mu_1,\phi_1) -\oL_n(K_1,\phi_1)\leq O( n^{-1}\log n)  $$
and  	$$ 0\leq \oL_n(\mu_2,\phi_2) -\oL_n(K_2,\phi_2)\leq O( n^{-1}\log n).  $$
The desired assertion follows from Theorem \ref{maintheorem1}.
\end{proof}

\medskip
\noindent\textbf{Funding.}
This work was funded by National Key R\&D Program of China 2025YFA1018300 and Basic Research Program of Jiangsu BK20251226.

\medskip
%\bibliographystyle{plain}
%\bibliography{refs}

\end{document}